\documentclass[11pt]{article}
\usepackage{bbm}
\usepackage{geometry}
\usepackage{color}
\usepackage{amsfonts}
\usepackage{latexsym, amssymb, amsmath, amscd, amsthm, amsxtra}
\usepackage{mathtools}
\usepackage{enumerate}
\usepackage[all]{xy}
\usepackage{mathrsfs}
\usepackage{fancyhdr}
\usepackage{listings}
\usepackage{hyperref}

\def\P{\mathbb{P}}

\def\E{\mathbb{E}}
\def\B{\operatorname{B}}

\def\11{\mathbbm{1}}
\def\Gc{\mathcal{G}}
\def\G{\mathsf G}
\def\Es{\mathsf E}
\def\ER{Erd\H{o}s-R\'enyi\ }

\def\Ps{\mathsf P}
\def\Qs{\mathsf Q}
\def\As{\mathsf A}
\def\Vs{\mathsf V}
\def\Gs{\mathsf G}

\newcommand{\Qc}{\mathcal Q}

\newcommand{\Pb}{\mathbb P}
\newcommand\dif{\mathop{}\!\mathrm{d}}

\newtheorem{DEF}{Definition}[section]
\newtheorem{thm}{Theorem}[section]

\newtheorem{proposition}[thm]{Proposition}

\newtheorem{lemma}[thm]{Lemma}

\theoremstyle{definition}
\newtheorem{remark}[thm]{Remark}

\numberwithin{equation}{section}

\begin{document}
	\title{Matching recovery threshold for correlated random graphs}
	
	\author{Jian Ding\\Peking University \and 
		Hang Du\\Peking University}

	\maketitle

	\begin{abstract}
For two correlated graphs which are independently sub-sampled from a common Erd\H{o}s-R\'enyi graph $\mathbf{G}(n, p)$, we wish to recover their \emph{latent} vertex matching from the observation of these two  graphs \emph{without labels}. When $p = n^{-\alpha+o(1)}$ for $\alpha\in  (0, 1]$, we establish a sharp information-theoretic threshold for whether it is possible to correctly match a positive fraction of vertices. Our result sharpens a constant factor in a recent work by Wu, Xu and Yu.
	\end{abstract}

\section{Introduction}\label{sec:intro}

In this paper, we study the information-theoretic threshold for recovering the latent matching between two correlated \ER graphs. To mathematically make sense of the problem, we first need to choose a model for a pair of correlated \ER graphs, and a natural choice is that the two graphs are independently sub-sampled from a common \ER graph. More precisely, for two vertex sets $V$ and $\mathsf V$ with cardinality $n$, let $E_0$ be the set of unordered pairs $(u,v)$ with $u,v\in V, u\neq v$ and define $\Es_0$ similarly with respect to $\mathsf V$. For some model parameters $p,s\in (0,1)$, define $\Qs$ to be the law for two (correlated) random graphs $G=(V,E)$ and $\mathsf G=(\mathsf V,\mathsf E)$ generated as follow: sample a uniform bijection $\pi^*:V\to \mathsf V$, independent Bernoulli variables $I_{(u,v)}$ with parameter $p$ for $(u,v)\in E_0$ as well as independent Bernoulli variables $J_{(u, v)}, \mathsf J_{(\mathsf u,\mathsf v)}$ with parameter $s$ for $(u, v)\in E_0$ and $(\mathsf u,\mathsf v)\in \Es_0$. Let 
\begin{equation}
	\label{eq:def-of-G_e}
	G_{(u,v)}=I_{(u,v)}J_{(u,v)},\forall (u,v)\in E_0\,,\quad \G_{(\mathsf u,\mathsf v)}=I_{\left((\pi^*)^{-1}(\mathsf u),(\pi^*)^{-1}(\mathsf v)\right)}\mathsf J_{(\mathsf u,\mathsf v)},\forall (\mathsf u,\mathsf v)\in \Es_0\,,
\end{equation}
and $E=\{e\in E_0:G_e=1\}, \mathsf E=\{\mathsf e\in \Es_0:\G_{\mathsf e}=1\}$. It is obvious that marginally $G$ is an \ER graph with edge density $ps$ (which we denote as $\mathbf{G}(n,ps)$) and so is $\mathsf G$.

A fundamental question is to recover the \emph{latent} matching $\pi^*$ from the observation of $(G, \mathsf G)$. More precisely, we wish to find an estimator $\hat \pi$ which is measurable with respect to $(G, \mathsf G)$ that maximizes $\operatorname{overlap}(\pi^*,\hat{\pi})$, where $\operatorname{overlap}(\pi^*, \hat \pi) = |\{v\in V: \pi^*(v) = \hat \pi(v)\}|$. Our main contribution is a sharp information-theoretic threshold for partial recovery, i.e., whether there exists a $\hat \pi$ such that $\operatorname{overlap}(\pi^*,\hat{\pi})\ge \delta n$ for some positive constant $\delta$.

\begin{thm}\label{thm:main}
	Suppose $p=n^{-\alpha+o(1)}$ for some fixed $\alpha\in(0,1]$ (where $o(1)$ denotes a term vanishing in $n$). Let $\lambda_*=\varrho^{-1}(\frac{1}{\alpha})$ (here $\varrho^{-1}$ is defined in Proposition~\ref{prop:densest-subgraph} below) and $\lambda=nps^2$. Then for any positive constant $\varepsilon$ the following hold: 
	\begin{itemize}
		\item If $\lambda\ge \lambda_*+\varepsilon$, then there exist an estimator $\hat{\pi}$ and a positive constant $\delta=\delta(\alpha,\varepsilon)$, such that \begin{equation}\label{eq:main_result_1}
			\Pb[\operatorname{overlap}(\pi^*,\hat{\pi})\ge \delta n]=1-o(1)\,.
		\end{equation}
		\item If $\alpha <1$ and $\lambda\le \lambda_*-\varepsilon$, then for any estimator $\hat{\pi}$ and any positive constant $\delta$, 
		\begin{equation}\label{eq:main_result_2}
			\Pb[\operatorname{overlap}(\pi^*,\hat{\pi})\ge \delta n]=o(1)\,.
		\end{equation}
	\end{itemize}
\end{thm}
We emphasize that \eqref{eq:main_result_2} in the case of $\alpha=1$ was shown in \cite{WXY21+}. While our proof should also be able to cover this case, we choose to focus on the case of $\alpha<1$ as this assumption helps avoiding some technical complications. As an important contribution, \cite{WXY21+} established a sharp threshold for $\alpha = 0$ and up-to-constant upper and lower bounds for $0 < \alpha \leq 1$. So in summary, our work is hugely inspired by \cite{WXY21+} and in return sharpens a constant factor therein and thus fills the remaining gap in the regime for $0 < \alpha \leq 1$.
 In addition, we note that the sharp threshold for exact recovery was established in \cite{WXY21+} which concerns the existence of $\hat\pi$ such that $\hat \pi = \pi^*$. 

\subsection{Background and related results}\label{subsec:Background}

Recently, there has been extensive study on the problem of matching the vertex correspondence between two correlated graphs and the closely related problem of detecting the correlation between two graphs. On the one hand, questions of this type have been raised from various applied fields such as social network analysis \cite{NS08,NS09}, computer vision \cite{CSS07,BBM05}, computational biology \cite{SXB08,VCP15} and natural language processing \cite{HNM05}; on the other hand, graph matching problems seem to provide another important set of examples which exhibit the intriguing \emph{information-computation gap}, whose theoretical analysis integrates tools from various branches of mathematics.

Despite the fact that Erd\H{o}s-R\'enyi Graph perhaps does not quite capture important features for any network arising from realistic problems, it is nevertheless reasonable to start our (presumably long) journey of completely understanding the information-computation phase transition for graph matching problems from a clean, simple and in some sense canonical random graph model such as Erd\H{o}s-R\'enyi. Along this line, many progress has been made recently, including information-theoretic analysis \cite{CK16, CK17, HM20, WXY20+,WXY21+} and proposals for various efficient algorithms  \cite{PG11, YG13, LFP14, KHG15, FQRM+16, SGE17, BCL19, DMWX21, BSH19, CKMP19,DCKG19, MX20, FMWX20,  GM20, FMWX19+, MRT21+, MWXY21+}.  As of now, it seems fair to say that we are still relatively far away from being able to completely understand the phase transition for computational complexity of graph matching problems. As in many other problems of this type, an information-theoretic phase transition is easier and usually will also guide the study on the transition for computational complexity. Together with previous works \cite{WXY20+, WXY21+, DD22+} (which were naturally inspired by earlier works such as \cite{CK16, CK17, HM20}), it seems now we have achieved a fairly satisfying understanding on the information-theoretic transition and we hope that this may be of help for future study on computational aspects.

As hinted from earlier discussions, an important (and in fact a substantially more important) research direction is to study graph matching problems on more realistic graph models other than  Erd\H{o}s-R\'enyi. This ambitious program has started seeing some progress, sometimes paralleling to that on Erd\H{o}s-R\'enyi graphs and sometimes inspired by insights accumulated on Erd\H{o}s-R\'enyi.
For instance, a model for correlated randomly growing graphs was studied in \cite{RS20+}, graph matching for correlated stochastic block model was studied in \cite{RS21} and graph matching for correlated random geometric graphs was studied in \cite{WWXY22+}. In a very recent work \cite{CJMNZ22+}, a related matching problem (albeit somewhat different from graph matching) was studied and it seems the method developed therein enjoyed direct and successful applications to single-cell problems.

\subsection{Connection to previous works}\label{subsec:connection-to-densest-subgraph}

We have learned from \cite{WXY20+, WXY21+} important insights on information thresholds for graph matching problems. An additional ingredient we realized is the connection to the densest subgraph, which allowed us to improve \cite{WXY20+} and establish the sharp detection threshold as in \cite{DD22+}. The densest subgraph problem arose in the study of load balancing problem \cite{Hajek90} and much progress has been made on densest subgraphs for random graphs \cite{CSW07, FR07, GW10, FKP16}. Of particular importance to us is the work of \cite{AS16} which in particular established the asymptotic value for the maximal subgraph density of an Erd\H{o}s-R\'enyi graph using the objective method from \cite{AS04}, as incorporated in the next proposition.
\begin{proposition}
	\label{prop:densest-subgraph}
	(\cite[Theorems 1 and 3]{AS16}, see also \cite[Propositions~2.1 and 2.3]{DD22+})
	There exists a continuous, strictly increasing and unbounded function $\varrho:[1,\infty)\to[1,\infty)$ (which can be explicitly characterized via a variational problem) with $\varrho(1)=1$, such that for any $\lambda\ge 1$ the maximal edge-vertex ratio over all nonempty subgraphs of an \ER graph $\mathcal H=(V,\mathcal E)\sim \mathbf{G}(n,\frac{\lambda}{n})$ concentrates around $\varrho(\lambda)$ as $n\to\infty$, i.e.,
	\begin{equation}
		\label{eq:density-concentration}
		\max_{\emptyset \neq U\subset V}\frac{|\mathcal E(U)|}{|U|}\to\varrho(\lambda)\mbox{ in probability as }n\to\infty\,.
	\end{equation}
	Furthermore, when $\lambda>1$, there is a constant $c_\lambda>0$ such that with probability tending to $1$ as $n\to\infty$, the densest subgraph in $\mathcal H$ (i.e. the maximizer of the left hand side of \eqref{eq:density-concentration}) has size at least $c_\lambda n$.
\end{proposition}
Denote $\varrho^{-1}:[1,\infty)\to[1,\infty)$ for the inverse function of $\varrho$. Building on insights from \cite{WXY20+} and using Proposition~\ref{prop:densest-subgraph}, we have established the following detection threshold in \cite{DD22+} as stated in the next proposition. Define $\Ps$ as the law of a pair of independent \ER graphs $\mathbf{G}(n,ps)$ on $V$ and $\Vs$, respectively. For two probability measures $\mu$ and $\nu$, we denote by $\operatorname{TV}(\mu, \nu) = \sup_{A}(\mu(A) - \nu(A))$ the total variation distance between $\mu$ and $\nu$.
\begin{proposition}
	\label{prop:d(P,Q)=o(1)}
	(\cite[Theorem 1.1]{DD22+})
	With notations in Theorem~\ref{thm:main}, the following hold:
	\begin{itemize}
		\item [(i)] If $\lambda\ge \lambda_*+\varepsilon$, then $\operatorname{TV}(\Ps,\Qs)=1-o(1)$;
		\item [(ii)] If $\alpha<1$ and $\lambda\le \lambda_*-\varepsilon$, then $\operatorname{TV}(\Ps,\Qs)=o(1)$.
	\end{itemize}
\end{proposition}
In order to prove Theorem~\ref{thm:main}, we combine insights from \cite{WXY21+} and the proof of Proposition~\ref{prop:d(P,Q)=o(1)}. It is perhaps not surprising that the partial recovery threshold coincides with the detection threshold. On the contrary, what may be unexpected at the first glance is that given \cite{DD22+} it still requires substantial amount of non-trivial work to prove impossibility of partial recovery as the default folklore assumption is that detection is easier than recovery. Indeed, detection is easier than exact recovery since with possibility of exact recovery one should be able to detect the correlation by examining the intersection of the two graphs  under this (estimated) matching. But if the estimator only achieves partial recovery, the intersection graph is not necessarily ``bigger'' than that of a random matching, which partly explains the difficulty for proving impossibility of partial recovery. As an analogy, the difficulty we face is similar to the challenge addressed in \cite{WXY21+} provided with \cite{WXY20+}, and it is possible that the additional difficulty is even more substantial for us since we need to nail down the exact threshold.
We refer the reader to the discussions at the beginning of Sections~\ref{sec:possibility} and \ref{sec:impossibility} for overviews of the proofs for \eqref{eq:main_result_1} and \eqref{eq:main_result_2}, respectively.

\subsection{Notations}
\label{subsec:notation-and-fact}
We record in this subsection a list of notations that we shall use throughout the paper. First recall that we have two vertex sets $V,\Vs$ with $|V|=|\Vs|=n$, and $\Ps,\Qs$ are two probability measures on pairs of random graphs on $V$ and $\Vs$ defined previously. In addition, for an edge $e\in E_0$, $G_e$ denotes for the indicator of the event that $e$ is an edge in $G$, and the similar notation applies for $\Gs_\mathsf e$ with $\mathsf e\in \Es_0$.
The following is a collection of notational conventions we shall follow.

$\bullet$ \emph{$\B(A,\As)$, $\B(V,\Vs)$ and $\B(V,\Vs,A,\sigma)$.} For any two sets $A\subset V$ and $\As\subset \Vs$, we denote $\operatorname{B}(A,\As)$ for the set of embeddings from $A$ to $\As$. In particular, $\B(V,\Vs)$ is the set of bijections from $V$ to $\Vs$. For any subset $A\subset V$ and any embedding $\sigma:A\to\Vs$, let $\B(V,\Vs,A,\sigma)\subset \B(V,\Vs)$ be the set of bijections which inhibit to $A$ as $\sigma$. 

$\bullet$ \emph{Induced subgraphs $H_A$ and $H^A$.} For a graph $H=(V,E)$ and a subset $A\subset V$, define $H_A=(A,E_A)$ to be the induced subgraph of $H$ in $A$, and $H^A=(V,E^A)$ to be the subgraph of $H$ obtained by deleting all edges within $A$. Similar notations $\mathsf H_\As=(\As,\Es_\As),\mathsf H^\As=(\Vs,\Es^\As)$ apply for any graph $\mathsf H=(\Vs,\Es)$ and any subset $\As\subset \Vs$.

$\bullet$ \emph{The probability measure $\Qc$.} Define a probability measure $\Qc$ on the space of triples 
\begin{equation}
	\label{eq:space-of-triples}
	\Omega=\{(\pi^*,G,\G):\pi^*\in \B(V,\Vs),G,\G\mbox{ are subgraphs of }(V,E_0),(\Vs,\Es_0)\}
\end{equation}
as follow: the marginal distributions of $\pi^*$ under $\Qc$ is uniform on $\B(V,\Vs)$; conditioned on $\pi^*$,  $(G,\Gs)$ is a pair of correlated \ER graphs given as in \eqref{eq:def-of-G_e}. It is clear that $\Qs$ is nothing but the marginal distribution for the last two coordinates of $\Qc$.

$\bullet$ \emph{Edge bijection and permutation.} For any bijection $\pi\in \B(V,\Vs)$ (respectively permutation $\phi$ on $V$), define the bijection between $E_0$ and $\Es_0$ (respectively permutation on $E_0$) induced by $\pi$ (respectively $\phi$) as $\Pi$ (respectively $\Phi$). That is to say, for any $(u,v)\in E_0$, we have $\Pi((u,v))=(\pi(u),\pi(v))$ (respectively $\Phi((u,v))=(\phi(u),\phi(v))$).

$\bullet$ \emph{Edge orbits and $\mathcal O_\pi$. }Assume $\pi^*\in \B(V,\Vs)$ is fixed. For any $\pi\in \B(V,\Vs)$, $\phi\stackrel{\operatorname{def}}{=}\pi^{-1}\circ \pi^*$ is a permutation on $V$. The induced edge permutation $\Phi$ on $E_0$ decomposes $E_0$ into disjoint edge cycles. We call these cycles as edge orbits induced by $\pi$ and denote $\mathcal O_\pi$ for the collection of all such edge orbits. The point of this notation is that, the families of random variables $\{(G_e,\G_{\Pi(e)}):e\in O\}$ are mutually independent under the law $\Qc[\cdot\mid \pi^*]$ for distinct edge orbits $O\in \mathcal O_\pi$.

$\bullet$ \emph{$\pi$-intersection graphs $\mathcal H_\pi=(V,\mathcal E_\pi)$.} For any triple $(\pi,G,\G)\in \Omega$ defined as in \eqref{eq:space-of-triples}, the $\pi$-intersection graph $\mathcal H_\pi=(V,\mathcal E_\pi)$ of $(G, \G)$ is a subgraph of $(V,E_0)$, where $(u,v)\in \mathcal E_\pi$ if and only if $(u,v)$ is an edge in $G$ and $(\pi(u),\pi(v))$ is also an edge in $\Gs$.

\medskip

\noindent {\bf Acknowledgement.} We warmly thank Nicholas Wormald, Yihong Wu and Jiaming Xu for stimulating discussions. Hang Du is partially supported by the elite undergraduate training program of School of Mathematical Sciences in Peking University.

\section{Possibility for partial recovery}\label{sec:possibility}

In this section, we prove Theorem~\ref{thm:main}-\eqref{eq:main_result_1}, where $\lambda\ge \lambda_*+\varepsilon$ for some arbitrary and fixed $\varepsilon > 0$. The construction of the estimator  $\hat{\pi}$ naturally takes inspiration from the detection statistics as in \cite[Section 2]{DD22+}. For instance, we may simply define $\hat{\pi}$ to be the maximizer for $\max_\pi \max_{U: |U| \geq c_\lambda n} \frac{|\mathcal E_\pi(U)|}{|U|}$ since this was shown in \cite{DD22+} as an efficient statistics for testing correlation against independence (or alternatively for convenience of analysis, choose $\hat {\pi}$ as an arbitrary matching whose intersection graph has maximal subgraph density exceeding a certain threshold). While this estimator may in fact has non-vanishing overlap with the true matching $\pi^*$, it seems rather difficult to prove as we now explain. In order to justify the estimator one has to show that typically any matching that has vanishing overlap with $\pi^*$ can not be a maximizer, and claims of this type are usually proved via a first moment computation. In \cite{DD22+} a first moment computation along this line was carried out to show that when two graphs are independent there is no matching whose intersection graph has subgraph with at least $c_\lambda n$ vertices and also a large edge density. In the independent case, all matchings are symmetric with each other (since graphs are generated independently with $\pi^*$) and thus the first moment computation is rather straightforward. However, in this paper the computation needs to be carried out for correlated graphs and thus overlapping structures with $\pi^*$ for different matchings play a significant role such that not only they complicate the computation but also it seems they actually will lead to a blowup of the first moment (but this does not necessarily imply that ``bad'' things do happen since the blowup may come from an event of small probability). One common approach in this case is to introduce further truncation, and a natural truncation that comes to mind is to pose an upper bound on the maximal subgraph edge density, since a plausible way for the first moment to blow up is due to an event of small probability where the maximal subgraph edge density on an intersection graph is excessively high. With these intuitions in mind, we define our estimator $\hat{\pi}=\hat{\pi}(G,\G)$ as follows.   
\begin{DEF}
	\label{def:estimator}
	Fix some $0<\eta<\frac{\varrho-\alpha^{-1}}{4}$, where $\varrho$ is short for $\varrho(\lambda)$. 
	For any matching $\pi\in \B(V,\Vs)$, we say it is a \emph{reasonable candidate} of $\pi^*$ if its $\pi$-intersection graph $\mathcal H_{\pi}$ of $(G,\G)$ satisfies the following two conditions:
	\begin{enumerate}[(i)]
		\item The edge-vertex ratio of any nonempty subgraph of $\mathcal H_{\pi}$ does not exceed $\varrho+\eta$;
		\item There is a subgraph of $\mathcal H_\pi$ with size at least $c_\lambda n$ and edge-vertex ratio at least $\varrho-\eta$.
	\end{enumerate}
	If the set of reasonable candidates is nonempty, choose one of them as $\hat{\pi}$ arbitrarily. Otherwise pick a $\hat{\pi}\in \B(V,\Vs)$ randomly. 
\end{DEF}
Note that $\mathcal H_{\pi^*}\sim \mathbf{G}(n,\frac{\lambda}{n})$  and thus by Proposition~\ref{prop:densest-subgraph} we see that $\pi^*$ is a reasonable candidate with probability tending to 1 as $n\to \infty$. Therefore, in order to prove \eqref{eq:main_result_1} it suffices to prove the following proposition.
\begin{proposition}
	\label{prop:union-bound-of-Bc}
	There exists $\delta= \delta(\alpha,\varepsilon) >0$ such that the following holds. Denote $\mathcal B$ for the event that there exists a reasonable candidate $\pi$ with $\operatorname{overlap}(\pi^*,\pi)\le \delta n$. Then $\Qc[\mathcal B]\to 0$ as $n\to\infty$.
\end{proposition}
We hope to bound $\Qc[\mathcal B]$ by the first moment method and we hope that Condition (i) for reasonable candidate will help controlling the moment. It was not \emph{a priori} clear why Condition (i) suffices, and in fact even after completing the proof we do not feel that there is a one-sentence explanation on why Condition (i) suffices since the proof of Proposition~\ref{prop:union-bound-of-Bc} seems to involve fairly nontrivial probability and combinatorics. 

A key estimate required to prove Proposition~\ref{prop:union-bound-of-Bc} is the tail probability for  $\mathcal H_\pi$ to satisfy Condition (ii).
Due to complication arising from correlations, this is not a very straightforward computation since on the one hand we only have independence between different edge orbits (see Section~\ref{subsec:prior-estimation} for definition) and on the other hand orbits with different sizes have different large deviation rates.  This motivates us to classify $\pi$ according to the structure of orbits (as in Section~\ref{subsec:prior-estimation}) and then perform a union bound over matchings via a union bound over different classes of matchings (as in Section~\ref{subsec:proof-of-union-bound}). Along the way, we will postpone proofs for a few technical lemmas/propositions into appendices to maintain a smooth flow of presentation.

\subsection{Orbits and tail probabilities}\label{subsec:prior-estimation}

In this subsection, for convenience we consider $\pi^*$ as fixed. Mathematically speaking, we condition on the realization of $\pi^*$ and we slightly abuse the notation by denoting $\pi^*$ as its realization. Fix some $A \subset V$. For any $\pi\in \B(V,\Vs)$, let $\phi\stackrel{\operatorname{def}}{=}\pi^{-1}\circ\pi^*$ and recall the definition of $\Pi$ and $\Phi$ in Section~\ref{subsec:notation-and-fact}. Similar with $\mathcal O_\pi$ in Section~\ref{subsec:notation-and-fact}, we also define a set of edge orbits $(\mathcal O_\pi)_A$ in $(E_0)_A$ as the collection of orbits induced by the mapping $\Phi$. More precisely, each orbit has the form $(e_1,\dots,e_k)$ with $e_1,\dots,e_{k}\in (E_0)_A,e_{i+1}=\Phi(e_i)$ for $1\le i\le k-1$ and in addition satisfies 
\begin{equation}\label{eq-cycle-chain}
	e_1=\Phi(e_k)\,,\quad\mbox{or}\quad\ \Phi^{-1}(e_1)\notin (E_0)_A,\Phi(e_k)\notin (E_0)_A\,.
\end{equation} It is clear that $(\mathcal O_\pi)_A$ is a partition of the edge set $(E_0)_A$. Again, the virtue of decomposing $(E_0)_A$ into orbits in $(\mathcal O_\pi)_A$ is that the families of random variables $\{(G_e,\mathsf G_{\Pi(e)}):e\in O\}$ are mutually independent under $\Qc[\cdot\mid\pi^*]$ for distinct $O\in (\mathcal O_\pi)_A$.

From the definition we see that any orbit $ O\in (\mathcal O_\pi)_A$ is a cycle (if the former occurs in \eqref{eq-cycle-chain}) or a chain (if the latter occurs in \eqref{eq-cycle-chain}), and we shall call $O$ a $k$-cycle (respectively $k$-chain) if it is a cycle (respectively chain) with length $k$ (i.e., with $k$ edges in the orbit). For convenience, we denote by $\mathrm{LCM}(x, y)$ as the least common multiple for two integers $x$ and $y$. The following lemma characterizes the structure of orbits in $(\mathcal O_\pi)_A$ in a more detailed way. The lemma is relatively obvious and an illustrative explanation can be found in \cite[Subsection 5.1]{WXY20+}. As a result, we omit its proof.
\begin{lemma}\label{lem-orbit-structure}
	For an edge $(u,v)\in (E_0)_A$, the following hold:\\
	\noindent (a) The orbit of $(u, v)$ in $(\mathcal O_{\pi})_A$ is a cycle if and only if the node cycle of $u$ with respect to  $\phi$  is entirely contained in $A$ and so is that for $v$.\\
	\noindent (b) If the node cycle of $u$ with respect to $\phi$ are disjoint from that of $v$ and both cycles are entirely contained in $A$ with lengths $x$ and $y$ respectively, then the orbit $O\in (\mathcal O_\pi)_A$ containing $(u,v)$ is a $\mathrm{LCM}(x, y)$-cycle.\\
	\noindent (c) If $u,v$  are in the same node cycle with respect to $\phi$ and if this cycle is entirely contained in $A$ with length $x$, then the orbit $O\in (\mathcal O_\pi)_A$ containing $(u,v)$ is an $x$-cycle or an $\frac{x}{2}$-cycle, with the latter happens only when $x$ is even and $v=\phi^{x/2}(u)$. The cycles in the latter case are called \emph{special}.
\end{lemma}
For $\pi\in \B(V,\Vs)$  and any edge orbit $O\in (\mathcal O_\pi)_A$, let $\mathcal E_O = O \cap \mathcal H_{\pi}$. We have the following exponential moments for $|\mathcal E_O|$ under the law $\Qc[\cdot\mid\pi^*]$. We write $a_n \ll b_n$ if $a_n/b_n \to 0$ as $n\to \infty$.
\begin{proposition}
	\label{prop:exponential-moment-of-orbit}
	For any $\theta$ with $1\ll e^\theta\ll n$, the following hold with $\mu_1=1+e^\theta /n^{1+\alpha+o(1)}$ and $\mu_2=(\lambda+o(1))e^\theta /n$:
	\begin{itemize}
		\item For any $k$-cycle $O_k\in (\mathcal O_\pi)_A$,
		\begin{equation}
			\label{eq:exp-moment-for-cycles}
			\mathbb{E}_{(G,\G)\sim\Qc[\cdot\mid\pi^*]}\exp\left({\theta|\mathcal{E}_{O_k}|}\right)=\mu_1^{k}+\mu_2^{k}\,.
		\end{equation}
		\item For any $k$-chain $O_k\in (\mathcal O_{\pi})_A$,
		\begin{equation}
			\label{eq:exp-moment-for-chains}
			\mathbb{E}_{(G,\G)\sim \Qc[\cdot\mid \pi^*]}\exp\left({\theta|\mathcal E_{O_k}|}\right)\le \mu_1^k+e^\theta n^{-1-2\alpha+o(1)}\mu_2^k\,.
		\end{equation}
	\end{itemize}
\end{proposition}

At this moment, we need to give a ``cutoff'' between short cycles and long cycles. To this end, we define
\begin{equation}
	\label{eq:N}
	N=
	\begin{cases}
		\lfloor(1-\alpha)^{-1}\rfloor,\ &\alpha<1\,,\\
		\lfloor(\varrho-\eta-1)^{-1}\rfloor+1,\ &\alpha=1\,.
	\end{cases}
\end{equation}
(We keep in mind that $\leq N$ means short and $> N$ means long.)
For a fixed $\pi\in \B(V,\Vs)$, 
let $E_{\operatorname{s}}$ be the total number of edges in $\mathcal{H}_\pi$ coming from special cycles in $(\mathcal O_\pi)_A$, and for $1\leq k\leq N$ let $E_k$ be the total number of edges in $\mathcal H_\pi$ coming from non-special $k$-cycles in $(\mathcal O_\pi)_A$, and let $E_{N+1}$ be the total number of edges in $\mathcal H_\pi$ coming from chains and non-special cycles in $(\mathcal O_\pi)_A$ with lengths at least $N+1$.
Proposition~\ref{prop:exponential-moment-of-orbit} leads to the following estimates on tail probabilities. 
\begin{lemma}
	\label{lem:ldp-rate}
	Let $\alpha_k = \frac{k-1}{k}$ for $1\leq k\leq N$ and let $\alpha_{N+1} =\alpha\wedge\frac{N}{N+1}$. For some $M=o(n\log n)$ and 
	any $0\le x\le 2\varrho n$, the following hold:
	\begin{align}
		&\Qc[E_{\operatorname{s}}\ge x\mid \pi^*]\le \exp(M-x\log n)\,,\label{eq:ldp-E_{-1}}\\
		&\Qc[E_k\ge  x\mid \pi^*]\le \exp\left(M-\alpha_kx\log n\right) \mbox{ for }1\le k\le N\label{eq:ldp-E_>0}\,,\\
		&\Qc[E_{N+1}\ge x\mid\pi^*]\le \exp(M-\alpha_{N+1} x\log n)\,.\label{eq:ldp-E_0} 
	\end{align}
\end{lemma}
\begin{proof}
	Suppose there are $S_k$ special $k$-cycles, $L_k$ non-special $k$-cycles and $T_k$ many $k$-chains in $(\mathcal O_\pi)_A$. Then $\sum_{k\ge 1}kS_k\le n$ since each vertex belongs to no more than one special cycle, and $\sum_{k\ge 1}k(L_k+T_k)\le n^2$ since the total number of edges in $(\mathcal O_\pi)_A$ is bounded by $n^2$. 
	
	By taking $\theta=\log n-\log\log n$ and applying Markov's inequality, the left hand side of \eqref{eq:ldp-E_{-1}} is bounded by
	\begin{equation}
		\label{eq:bound-ldp-1}
		\begin{aligned}
			&e^{-\theta x}\prod_{k\ge1}\left(\exp(\theta|\mathcal E_{O_k}|)\right)^{S_k}\stackrel{\eqref{eq:exp-moment-for-cycles}}{=}e^{-\theta x}\prod_{k\ge1}\left(\mu_1^k+\mu_2^k\right)^{S_k}
			\le e^{-\theta x}(\mu_1+\mu_2)^{\sum_{k\ge 1}kS_k}\\
			\le&\ e^{-\theta x}\left[1+\frac{e^\theta}{n^{1+\alpha+o(1)}}+\frac{(\lambda+o(1)) e^\theta}{n}\right]^n
			\leq \exp(o(n)-x\log n+x\log \log n)\,.
		\end{aligned}
	\end{equation}
	
	Similarly, by taking $\theta=\alpha_k\log n-\log \lambda$ and applying Markov's inequality together with \eqref{eq:exp-moment-for-cycles} again, we see the left hand side of \eqref{eq:ldp-E_>0} is bounded by
	\begin{equation}
		\label{eq:bound-ldp-2}
		\begin{aligned}
			e^{-\theta x}\left[\mu_1^k+\mu_2^k\right]^{L_k}
			\le&\ e^{-\theta x}\left(1+\frac{e^\theta}{n^{1+\alpha+o(1)}}+\frac{(\lambda^k+o(1))e^{k\theta}}{n^k}\right)^{n^2}\\\le&\ \exp(x\log\lambda+n+o(n)-\alpha_kx\log n)\,.
		\end{aligned}
	\end{equation}
	
	Finally, applying Markov's inequality together with \eqref{eq:exp-moment-for-cycles} and  \eqref{eq:exp-moment-for-chains}, we see for any $\theta$  with $1\ll e^\theta\ll n$, the left hand side of \eqref{eq:ldp-E_0} is bounded by
	\begin{equation}
		\label{eq:bound-ldp-3}
		\begin{aligned}
			&\ e^{-\theta x}\prod_{k\ge N+1}\left(\mu_1^k+\mu_2^k\right)^{L_k}\prod_{k\ge 1}\left(\mu_1^k+e^\theta n^{-1-2\alpha+o(1)}\mu_2^k\right)^{T_k}\\
			\le &\ e^{-\theta x}\prod_{k\ge N+1}\left(\mu_1^{N+1}+\mu_2^{N+1}\right)^{\frac{kL_k}{N+1}}\prod_{k\ge 1}\mu_1^{kT_k}(1+e^\theta n^{-1-2\alpha+o(1)}\mu_2^k)^{T_k}\\ 
			\le &\ e^{-\theta x}\left(1+e^\theta n^{-1-\alpha+o(1)}+(\lambda^{N+1}+o(1))e^{(N+1)\theta}n^{-N-1}\right)^{n^2}\left(1+e^{2\theta}n^{-2-2\alpha+o(1)}\right)^{n^2}\,.
		\end{aligned}
	\end{equation}
	When $\alpha<1$, we pick $\theta$ such that $e^\theta n^{-1-\alpha+o(1)}$ in the first bracket above equals to $n^{-1}$, then \eqref{eq:bound-ldp-3} becomes $\exp(n+o(n)-(\alpha-o(1))x\log n)$, and here crucially we used that $N+1>(1-\alpha)^{-1}$, implying $n^{(N+1)(\alpha - o(1))-N-1}\ll n^{-1}$. When $\alpha=1$, we just pick $\theta=\alpha_{N+1}\log n-\log\lambda$ and \eqref{eq:bound-ldp-3} becomes $\exp(x\log \lambda+n+o(n)-\alpha_{N+1}x\log n)$.
	
	Take $M=(2\varrho+1)(n\log\log n+o(1)n\log n)=o(n\log n)$ (with a suitable choice of $o(1)$ originating from $o(1)$-terms as above). Then with all of the aforementioned bounds, \eqref{eq:ldp-E_{-1}}, \eqref{eq:ldp-E_>0} and \eqref{eq:ldp-E_0} hold for large $n$, as desired.
\end{proof}

\subsection{Proof of Proposition~\ref{prop:union-bound-of-Bc}}\label{subsec:proof-of-union-bound}

For  $A\subset V$ and any $\sigma\in \B(A,\Vs)$, choose some $\bar\sigma\in \B(V,\Vs,A,\sigma)$ as an extension of $\sigma$ on $V$. It is easy to see that the set $(\mathcal O_{\bar\sigma})_A$ does not depend on the choice of extension $\bar\sigma$, and hence the set of node cycles of $\phi=\bar\sigma^{-1}\circ\pi^*$ which are entirely contained in $A$ is also well-defined. For any sequence of non-negative integers $n_1,\dots,n_N$ satisfying $\sum_{k=1}^{N}kn_k\le |A|$,  we define  $\operatorname{S}(A,n_1,\dots,n_N)\subset \B(A,\Vs)$ to be
\begin{equation*}
	\{\sigma\in \B(A,\Vs):\phi=\bar\sigma^{-1}\circ \pi^*\mbox{ has }n_k\mbox{ node cycles with length }k\mbox{ in }A \mbox{ for }1\le k\le N\}\,.
\end{equation*}
\begin{lemma}
	\label{lem:number-of-permutations-with-many-small-cycle}
	For $A\subset V$ with $|A|=T$ and non-negative integers $n_1,\dots,n_N$ satisfying $\sum_{k=1}^N kn_k\le T$, it holds that
	\begin{equation*}
		|\operatorname{S}(A,n_1,\dots,n_N)|\le\frac{ n(n-1)\cdots(n-T+1)}{\prod_{k=1}^N k^{n_k}n_k!}= \exp(O(n)+(T-n_1-\dots-n_N)\log n)\,.
	\end{equation*}
\end{lemma}
\begin{proof}
	Define 
	\[
	\operatorname{S}=\bigcup_{\sigma\in \operatorname{S}(A,n_1,\dots,n_N)}\B(V,\Vs,A,\sigma)
	\]
	to be the collection for all extensions $\pi\in \B(V,\Vs)$ of some $\sigma\in \operatorname{S}(A,n_1,\dots,n_N)$. This is a disjoint union since for $\pi_1, \pi_2$ from different $\B(V,\Vs,A,\sigma_1), \B(V,\Vs,A,\sigma_2)$, we have $\pi_1|_A=\sigma_1\neq \sigma_2=\pi_2|_A$, implying $\pi_1\neq \pi_2$ (here we denote by $\pi|_A$ the restriction of $\pi$ on $A$). In addition, note that for any $\pi\in \operatorname{S}$, $\phi=\pi^{-1}\circ \pi^*$ contains at least $n_k$ many $k$-node cycles for all $1\le k\le N$. By \cite[Theorem 1]{AT92} (see also \cite[Lemma 13]{WXY21+}), the number for such $\phi$ is no more than $\frac{n!}{\prod_{k=1}^N k^{n_k}n_k!}$. As a result, 
	\begin{equation}
		|\operatorname{S}|=(n-T)! \cdot |S(A,n_1,\dots,n_N)|\le \frac{n!}{\prod_{k=1}^N k^{n_k} n_k!}\,,
	\end{equation}
	which yields the desired inequality (in the lemma statement) and the equality follows from Stirling's formula.
\end{proof}

\begin{proof}[Proof of Proposition~\ref{prop:union-bound-of-Bc}] 
	For any  $A \subset V$ with $|A|=T\in [c_\lambda n,n]$ and any embedding $\sigma\in \B(A,\Vs)$, denote $\mathcal B_\sigma$ for the event that the $\bar\sigma$-intersection graph $\mathcal H_{\bar\sigma}$ satisfies the following:\\
	\noindent (i)  $|\mathcal E_{\bar\sigma}(A)|\ge (\varrho-\eta)T$;\\
	\noindent (ii) $|\mathcal E_{\bar\sigma}(U)|\le (\varrho+\eta){|U|}$ for any $U\subset A$.\\
	Again, we note that $\mathcal B_\sigma$ does not depend on the choice of extension $\bar\sigma$.
	
	For each $\pi^*\in \B(V,\Vs)$, it is clear that conditioned on $\pi^*$, $\mathcal B$ implies $\mathcal B_\sigma$ happens for some $A$ with size at least $c_\lambda n$ and some $\sigma\in \B(A,\Vs)$ which agrees with $\pi^*$ on less than $\delta n$ vertices in $A$. Then a simple union bound yields that 
	\begin{align}
		\label{eq:union-bound-0}
		&\ \Qc[\mathcal B\mid \pi^*]
		\le\ \sum_{T\ge c_\lambda n}\sum_{|A|=T} \sum_{n_1, \ldots, n_N}  \sum_{\sigma\in \operatorname{S}(A,n_1,\dots,n_N)}\Qc[\mathcal B_\sigma\mid \pi^*] \nonumber\\
		\le &\ \sum_{|T|\ge c_\lambda n}\sum_{|A|=T}\sum_{n_1,\dots,n_N}|\operatorname{S}(A,n_1,\dots,n_N)|\times\sup_{\sigma\in \operatorname{S}(A,n_1,\dots,n_N)}\Qc[\mathcal B_\sigma\mid \pi^*]\,,
	\end{align}
	where the summation for $n_i$'s is over all non-negative integers $n_1,\dots,n_N$ with $\sum_{k=1}^N kn_k\le T$ and $n_1\le \delta n$ (since $\sigma$ overlaps on less than $\delta n$ vertices with $\pi^*$).
	
	We next turn to bound $\sup_{\sigma\in \operatorname{S}(A,n_1,\dots,n_N)}\Qc[\mathcal B_\sigma\mid \pi^*]$. For any fixed $T,A,n_1,\dots,n_N$ and $\sigma\in \operatorname{S}(A,n_1,\dots,n_N)$, let $A_i\subset A$ be the set of vertices in node cycles of $\phi=\bar\sigma^{-1}\circ \pi^*$ with length $i$ in $A$. Then $|A_i|= i n_i$ since there are $n_i$ many $i$-cycles in $A$. Let $x_{ij}$ be the number of edges in $\mathcal H_\pi$ with two end points in $A_i$ and $A_j$ respectively, but not in special cycles. Then by Lemma~\ref{lem-orbit-structure},
	\[
	E_k=\sum_{i,j: \mathrm{LCM}(i,j)=k}x_{ij}, \mbox{ for all } 1\le k\le N\,.
	\] 
	Now from (ii) in $\mathcal B_\sigma$, we get for all $1\le m\le N$,
	\begin{equation}
		\sum_{k=1}^mE_k=\sum_{i,j:\mathrm{LCM}(i,j)\le m}x_{ij}\le \left| \mathcal E_{\bar\sigma}\left(A_1\cup A_2\cup\cdots\cup A_m\right)\right|\le (\varrho+\eta)\sum_{k=1}^m kn_k\,.
	\end{equation}
	Combined with (i) in $\mathcal B_\sigma$, this motivates us to define
	\begin{align*}
		&\Sigma_T=\left\{(x_0,\dots,x_{N+1})\in \mathbb{Z}^{N+2}: 0\le x_0,\dots,x_{N+1}\le \varrho T\,,
		\sum_{k=0}^{N+1}x_k\ge(\varrho-\eta)T
		\right\}\,,\\
		&\Delta_{n_1,\dots,n_N}=\left\{(x_0,\dots,x_{N+1})\in \mathbb{Z}^{N+2}:\sum_{k=1}^m x_k\le (\varrho+\eta)\sum_{k=1}^m kn_k,\forall 1\le m\le N\right\}\,.
	\end{align*}
	Somewhat mysteriously, the restriction to $\Delta_{n_1,\dots,n_N}$ which originates from our further truncation as in (ii) of $\mathcal B_\sigma$ suffices to control the (otherwise) blowup of the first moment. The seemingly computational coincidence is encapsulated in the following (purely algebraic) lemma. Denote $M(T,n_1,\dots,n_N)$ as
	\begin{equation}
		\min\left\{\sum_{k=1}^N n_k-T+x_0+\sum_{k=1}^{N+1}\alpha_kx_k\ \Bigg|\  (x_0,\dots,x_{N+1})\in \Sigma_T\cap \Delta_{n_1,\dots,n_N}\right\}\,.
	\end{equation}
	In light of Lemmas~\ref{lem:ldp-rate} and \ref{lem:number-of-permutations-with-many-small-cycle}, $M(T,n_1,\dots,n_N)$ captures the tradeoff between probability and enumeration.
	\begin{lemma}
		\label{lem:compute-the-minimum}
		There exist constants $\delta, \delta_0>0$, such that for any $T\in [c_\lambda n,n]$ and $n_1,\dots,n_N$ satisfying $n_1\le \delta n$ and $n_1+2n_2+\dots+Nn_N\le T$, we have that $M(T,n_1,\dots,n_N)\ge \delta_0 T$.
	\end{lemma}
	With Lemma~\ref{lem:compute-the-minimum} at hand, we see the supremum of $\Qc[\mathcal B_\sigma\mid \pi^*]$ for $\sigma\in \operatorname{S}(A,n_1,\dots,n_N)$ is bounded by
	\begin{align}
		\label{eq:union-bound-1}
		&\ \sum_{(x_0,\dots,x_{N+1})\in \Sigma_{T}\cap \Delta_{n_1,\dots,n_N}}\Qc[E_{\operatorname{s}}\geq x_0 \mbox{ and } E_{k}\ge x_k,1\le k\le N\mid \pi^*] \nonumber\\
		\stackrel{\text{independence}}{=}&\ \sum_{(x_0,\dots,x_{N+1})\in \Sigma_{T}\cap \Delta_{n_1,\dots,n_N}} \Qc[ E_{\operatorname{s}}\ge x_0\mid \pi^*]\prod_{k=1}^{N+1}\Qc[ E_k\ge x_k\mid \pi^*] \nonumber\\
		\stackrel{\text{Lemma~\ref{lem:ldp-rate}}}{\le} &\ \sum_{(x_0,\dots,x_{N+1})\in\Sigma_{T}\cap \Delta_{n_1,\dots,n_N}}\exp\left(o(n\log n)-\left[x_0+\sum_{k=1}^{N+1} \alpha_kx_k\right] \log n\right)\,.
	\end{align}
	Plugging \eqref{eq:union-bound-1} and Lemma~\ref{lem:number-of-permutations-with-many-small-cycle} into \eqref{eq:union-bound-0}, we see $\Qc[\mathcal B\mid\pi^*]$ is bounded by
	\begin{align}
		&\nonumber\sum_{T\ge c_\lambda n}\sum_{|A|=T}\sum_{n_1,\dots,n_N}\sum_{(x_0,\dots,x_{N+1})\in \Sigma_T\cap \Delta_{n_1,\dots,n_N}}\\
		&\nonumber\quad\quad\quad\exp\left(o(n\log n) -\left[\sum_{k=1}^N n_k-T+x_{0}+\sum_{k=1}^{N+1}\alpha_kx_k\right]\log n\right)\nonumber\\
		\le&\ \sum_{T\ge c_\lambda n}\binom{n}{T}T^N(\varrho n)^{N+2}\exp\left(o(n\log n)-M(T,n_1,\dots,n_N)\log n\right)\,.\label{eq:union-bound-2}
	\end{align}
By Lemma~\ref{lem:compute-the-minimum}, we see \eqref{eq:union-bound-2} is no more than 
	$$
	 \sum_{T\ge c_\lambda n}\exp(o(n\log n)-\delta_0 T\log n)=o(1)\,.
	 $$
	Thus, $\Qc[\mathcal B\mid \pi^*]=o(1)$. Since this is invariant for any $\pi^* \in \B(V, \Vs)$, we complete the proof of \eqref{eq:main_result_1}.
\end{proof}

\section{Impossibility for partial recovery}\label{sec:impossibility}

In this section we prove Theorem~\ref{thm:main}-\eqref{eq:main_result_2}.  Recall that we are now in the regime of $\lambda\le\lambda_*-\varepsilon$ for some arbitrary and fixed $\varepsilon>0$. At the first glance, this might seem trivial (as the authors have wrongly speculated) in light of Part (ii) in Proposition~\ref{prop:d(P,Q)=o(1)}. However, it turns out to be not at all obvious how to derive impossibility for correctly matching a positive fraction of vertices from impossibility of detection, since for instance we may correctly match a linear sized independent set in $\mathcal H_{\pi^*}$ but the intersection graph for this matching can be similar to that of a typical matching for independent graphs (and thus has no power for detection).  

Fix an arbitrary $\delta>0$. For a pair of graphs $(G,\G)$, denote $\Qc_{G,\G}$ for the posterior distribution of $\pi^*$ under the law $\Qc$ when $(G,\G)$ are given as observations. If there exists an estimator $\hat{\pi}=\hat{\pi}(G,\G)$ such that $\hat{\pi}$ correctly matches at least a $\delta$-fraction of vertices with non-vanishing probability, then $\Qc_{G,\G}$ must be somehow concentrate around $\hat{\pi}$ with non-vanishing probability. In light of this, for any $\tilde{\pi}\in \B(V,\Vs)$, we define
\begin{equation}
M(G,\G,\tilde\pi)=\Qc_{G,\G}[ \operatorname{overlap}(\pi^*,\tilde{\pi})\ge \delta n]=\frac{1}{\Qs[G,\G]}\sum_{\pi:\operatorname{overlap}(\pi,\tilde{\pi})\ge \delta n}\Qc[\pi,G,\G]
\end{equation}
to be the measure of $\Qc_{G,\G}$ on the set of $\pi\in\B(V,\Vs)$ which agrees with $\tilde{\pi}$ on more than $\delta n$ vertices. Further, we set
\begin{equation}
W(G,\G)=\max_{\tilde{\pi}\in \B(V,\Vs)}M(G,\G,\tilde{\pi})\,.
\end{equation}
Then the following proposition is the key to the proof of \eqref{eq:main_result_2}. 
\begin{proposition}
	\label{prop:reduce-to-posterior}
	For any $\delta>0$, we have $\mathbb{E}_{(G,\G)\sim \Qs}W(G,\G)\to 0$ as $n\to\infty$.
\end{proposition}
\begin{proof}[Proof of Theorem~\ref{thm:main}-\eqref{eq:main_result_2} assuming Proposition~\ref{prop:reduce-to-posterior}]
	 For any estimator  $\hat{\pi}=\hat \pi(G,\G)$ (perhaps with additional randomness), denote $\Pb_{G,\mathsf G}$ for the law of $\hat{\pi}$ given $(G,\G)$. Then the probability that $\operatorname{overlap}(\pi^*,\hat{\pi})\ge \delta n$ can be expressed as
	\[
	\mathbb{E}_{(G,\G)\sim \Qs}\mathbb{E}_{\hat{\pi}\sim \Pb_{G,\G}}\Qc_{G,\G}[\operatorname{overlap}(\pi^*,\hat{\pi})\ge \delta n]\le \mathbb{E}_{(G,\G)\sim \Qs}W(G,\G)\,,
	\]
	which is $o(1)$ by Proposition~\ref{prop:reduce-to-posterior}. This completes the proof.
\end{proof}
The starting point for the proof of Proposition~\ref{prop:reduce-to-posterior} is the observation that for any $K \leq \delta n$ overlapping on more than $\delta n$ vertices necessarily implies overlapping on some $A \subset V$ with $|A| = K$,
and indeed we will choose $K = n^\beta$ for $\beta$  close to $1$ but less than $1$. On the one hand, we choose $\beta$ close to 1 to avoid losing too much information; on the other hand, we choose $\beta$ strictly less than 1 since it  seems second moment computation applies better when $K$ is smaller. The second moment computation also leads to the next very simple but useful observation that $\max_{x\in S}x^2\le \sum_{x\in S}x^2$, which then allows us to upper-bound the maximum of posterior probability by its second moment. At the first glance, this bound seems really loose to be useful, but it turns out that for $K \ge n^\beta$ with some $\beta$ very close to $1$, applying this simple inequality does yield an efficient upper bound on the probability of correctly matching vertices in $A$ (see \eqref{eq-K-not-too-small}). Therefore, the main technical obstacle is now reduced to a second moment computation (see Proposition~\ref{prop:second-moment-bound}) for which we take inspiration from \cite{DD22+} and we remark here that a few additional truncations (on top of those in \cite{DD22+}) are necessary in order for our purpose. Finally, we note that in light of Proposition~\ref{prop:d(P,Q)=o(1)}, we can carry out many computations under the measure $\mathsf P$ for independent graphs, which in many ways simplifies our analysis (see e.g., Proposition~\ref{prop:reduce-to-posterior-truncation} and Proof of Proposition~\ref{prop:reduce-to-posterior} assuming Proposition~\ref{prop:reduce-to-posterior-truncation}).

\subsection{Truncations for the second moment}\label{subsec:good-event-and-good-set}

As mentioned earlier, we necessarily need to introduce truncations in order to prevent the second moment from blowing up. This truncation is expressed as some ``good'' event $\Gc$ as in Definition~\ref{def:good-event}  below, and as we will see $\Gc$ is measurable with respect to $(\pi^*,G,\G)$ and satisfies $\Qc[\Gc]\to 1$ as $n\to\infty$. Thus, it would be useful to reduce Proposition~\ref{prop:reduce-to-posterior} to a version with truncation as follows. 
\begin{proposition}\label{prop:reduce-to-posterior-truncation}
For any $\delta>0$, we have
	\begin{equation}\label{eq:reduction_2}
		\mathbb{E}_{(G,\G)\sim \Ps}\frac{1}{\Ps[G,\G]}\max_{\tilde{\pi}\in \B(V,\Vs)}\sum_{\pi:\operatorname{overlap}(\pi,\tilde{\pi})\ge \delta n}\Qc[\pi,G,\G]\mathbf{1}_\Gc\to 0,\mbox{ as }n\to\infty\,.
	\end{equation}
\end{proposition}
\begin{proof}[Proof of Proposition~\ref{prop:reduce-to-posterior} assuming Proposition~\ref{prop:reduce-to-posterior-truncation}]
	For any $\tilde{\pi}\in\B(V,\Vs)$, we can upper-bound $M(G,\G,\tilde{\pi})$ by 
	\begin{equation*}
		\frac{1}{\Qs[G,\G]}\sum_{\pi:\operatorname{overlap}(\pi,\tilde{\pi})\ge \delta n}\Qc[\pi,G,\G]\mathbf{1}_{\Gc}+\frac{\Qc[\Gc^c,G,\G]}{\Qs[G,\G]}\,,
	\end{equation*}
	where $\Qc[\Gc^c,G,\G] = \sum_\pi \Qc[\pi, G, \G]\mathbf 1_{\Gc^c}(\pi, G, \G)$.
This yields that
	\begin{equation}
		\label{eq:reduction_1}
		\mathbb{E}_{(G,\G)\sim \Qs}W(G,\G)\le \frac{1}{\Qs[G,\G]}\max_{\tilde{\pi}\in \B(V,\Vs)}\sum_{\pi:\operatorname{overlap}(\pi,\tilde{\pi})\ge\delta n}\Qc[\pi,G,\G]\mathbf{1}_\Gc+\Qc[\Gc^c]\,.
	\end{equation}
	
	Recall the definition of $\Ps$ and the fact that $\operatorname{TV}(\Ps,\Qs)=o(1)$ given in Proposition~\ref{prop:d(P,Q)=o(1)}-(ii). Define the event $\Gc_0$ as
	\[
	\Gc_0=\left\{(G,\G):\frac{\Qs[G,\G]}{\Ps[G,\G]}\ge \frac{1}{2}\right\}\,.
	\]
	Since $(G,\G)\in \Gc_0^c$ implies $\Qs[G,\G]\le \Ps[G,\G]-\Qs[G,\G]$, we have that 
	\begin{equation}
		\label{eq:Gc_0-is-o(1)}
			\Qs[\Gc_0^c]\le\int_{\Gc_0^c}\left(\dif \Ps-\dif \Qs\right)\le \operatorname{TV}(\Ps,\Qs)= o(1)\,.
	\end{equation}

	Since $\frac{1}{\Qs[G,\G]}\sum_{\pi:\operatorname{overlap}(\pi,\tilde{\pi})\ge \delta n}\Qc[\pi,G,\G]\le 1$ holds for any $\tilde{\pi}\in\B(V,\Vs)$, we can upper-bound the right hand side of \eqref{eq:reduction_1} by
	\begin{equation*}
		\begin{aligned}
			&\ \mathbb{E}_{(G,\G)\sim \Qs}\frac{\mathbf 1_{\Gc_0}(G,\G)}{\Qs[G,\G]}\max_{\tilde{\pi}\in \B(V,\Vs)}\sum_{\pi:\operatorname{overlap}(\pi,\tilde{\pi})\ge \delta n}\Qc[\pi,G,\G]\mathbf{1}_{\Gc}+\Qs[\Gc_0^c]+\Qc[\Gc^c]\\
			\le &\ \mathbb{E}_{(G,\G)\sim \Qs}\frac{2}{\Ps[G,\G]}\max_{\tilde{\pi}\in\B(V,\Vs)}\sum_{\pi:\operatorname{overlap}(\pi,\tilde{\pi})\ge \delta n}\Qc[\pi,G,\G]\mathbf{1}_{\Gc}+\Qs[\Gc_0^c]+\Qc[\Gc^c]\\
			\le &\ \mathbb{E}_{(G,\G)\sim \Ps}\frac{2}{\Ps[G,\G]}\max_{\tilde{\pi}\in \B(V,\Vs)}\sum_{\pi:\operatorname{overlap}(\pi,\tilde{\pi})\ge \delta n}\Qc[\pi,G,\G]\mathbf{1}_{\Gc}+2\operatorname{TV}(\Ps,\Qs)+\Qs[\Gc_0^c]+\Qc[\Gc^c]\,,
		\end{aligned}
	\end{equation*}
which is $o(1)$ by Proposition~\ref{prop:reduce-to-posterior-truncation}, \eqref{eq:Gc_0-is-o(1)} and the choice of $\Gc$. This concludes the proof.
\end{proof}

Next, we give the precise definition of the good event $\Gc$ and further introduce the concept of good set. These definitions may be a bit perplexing at the first glance, but the motivations behind all the constraints would become clear in
Section~\ref{sec:proof-of-second-moment-bound} (a reader may skip this definition for now and reference back when certain constraints are used in controlling moments in later proofs). 
We first introduce some constants $\xi,\zeta,\beta,C,\delta_1$. Denote 
\begin{equation}\label{eq:xi}
	\xi=\frac{\varrho+\alpha^{-1}}{2}\,,
\end{equation} 
where $\varrho$ is short for $\varrho(\lambda)$. Note that under the assumption  $\lambda\le\lambda^*-\varepsilon$, we have $\xi<\alpha^{-1}$. Since $\alpha < 1$, it is possible to take a constant $\zeta>1$ such that
\begin{equation}\label{eq:zeta}
	1+\zeta(\alpha-1)<2-\zeta\,.
\end{equation} 
 Then, we choose $\beta$ such that 
\begin{equation}\label{eq:beta}
(1-\alpha)\vee\frac{1+\zeta(\alpha-1)}{2-\zeta}<\beta<1\,,
\end{equation}
and some large integer $C$ such that 
\begin{equation}\label{eq:C}
	\alpha(\xi+C^{-1})<1\,.
\end{equation} 
Finally, we pick $\delta_1>0$ such that
\begin{equation}\label{eq:delta}
	\delta_1<(1-\alpha\xi)\wedge C^{-1}\beta\,.
\end{equation} 
\begin{DEF}\label{def:good-event}
	For a graph $\mathcal{H}=(V,\mathcal E)$, we say $\mathcal H$ is \emph{admissible} if $\mathcal H$ satisfies the following properties:
	\begin{enumerate}[(i)]
		\item For any $A\subset V$, it holds
		$
	    |\mathcal{E}(A)|\le \xi|A|
		$;
	    \item For any $A\subset V$ satisfying $|A|\le n/\log n$, it holds
	    $
	    	|\mathcal{E}(A)|\le \zeta |A|
	    $;
		\item The maximal degree of $\mathcal H$ is less than $\log n$;
		\item Any connected subgraph of $\mathcal H$ with size less than $\log \log n$ contains at most one cycle; 
		\item For any $k\ge 3$, the number of cycles with length $k$ in $\mathcal H$ is bounded by $n^{\delta_1 k}$.
	\end{enumerate}
Let $\Gc$ be the event that $\mathcal H_{\pi^*}$, the $\pi^*$-intersection graph of $G$ and $\G$, is admissible.
\end{DEF}
It is clear that $\Gc$ is measurable with respect to the triple $(\pi^*,G,\G)$. We next show that $\Gc$ is indeed a typical event under $\Qc$.
\begin{lemma}
	\label{lem:G-is-a.s.s.}
	$\Qc[\Gc]\to 1$ as $n\to\infty$.
\end{lemma}
\begin{proof}
	Note that $\mathcal H_{\pi^*}$ has the law of an Erd\H{o}s-R\'enyi graph $\mathbf G(n,\frac\lambda n)$, it suffice to bound the probability that either of (i)-(v) fails for such an Erd\H{o}s-R\'enyi graph. $\Pb[(i) \mbox{ fails}]=o(1)$ follows from Proposition~\ref{prop:densest-subgraph}. For (ii), since $\zeta>1$, we can take a union bound as (denoting by $\mathbf B(m, q)$ as a binomial variable with $m$ trials and success probability $q$)
	\begin{equation*}
		\begin{aligned}
		\Pb[(ii)\mbox{ fails}]\le&\ \sum_{k\le n/\log n}\binom{n}{k}\Pb\left[\mathbf B\left(\binom{k}{2},\frac{\lambda}{n}\right)>\zeta k\right]\\
		\le &\ \sum_{k\le n/\log n}\binom{n}{k}\exp\left[-\zeta k\log\left( \frac{\zeta k}{\binom{k}{2}\frac{\lambda}{n}}\right)+\zeta k\right]\\
		=&\ \sum_{k\le n/\log n}\exp\left[(1-\zeta)k\log\left(\frac{n}{k}\right)+O(k)\right]=o(1)\,.
		\end{aligned}
	\end{equation*} 
$\P[(iii) \mbox{ fails}]=o(1)$ and $\P[(iv)\mbox{ fails}]=o(1)$ are well-known: indeed, the typical value of maximal degree of a $\mathbf G(n,\frac\lambda n)$ graph is of order $\log n/\log \log n$ (see e.g. \cite[Theorem 3.4]{AM22}), and the typical value for the minimal size of connected subgraphs containing more than one cycles is of order $\log n$ (the upper-bound follows readily from a union bound). Finally, for (v), since the expected number of $k$-cycles in a $\mathbf G(n, \frac\lambda n)$ graph is bounded by $\lambda^k$ for any $k\ge 3$, by Markov inequality we get
\[
\Pb[(v)\mbox{ fails}]\le \sum_{k=3}^{\infty}\frac{\lambda^k}{n^{\delta_1 k}}=o(1)\,.
\]
Altogether, we conclude the lemma.
\end{proof}
To make another layer of truncation, we introduce the concept of \emph{good set}: on the one hand, good set is abundant as shown in Lemma~\ref{lem:good-set}; on the other hand, good set will help controlling both the enumeration of embeddings as in Lemma~\ref{lem-apply-good-set} and the number of edges in certain subgraphs as in Lemma~\ref{lem:apply-good-set-2}.
\begin{DEF}\label{def:good-set}
	For a graph $\mathcal H=(V,\mathcal E)$, let $\operatorname{d}_\mathcal{H}$ be the graph distance on $\mathcal H$. We say a set $A\subset V$ is a \emph{good set} in $\mathcal H$, if it satisfies the following two properties:
	\begin{itemize}
	    \item For any two vertices $u,v\in A$, $\operatorname{d}_\mathcal{H}(u,v)>2C+2$. 
		\item For any vertex $w\in A$ and any $k$-cycle $\mathcal C$ in $\mathcal H$ with $k\le C$, we  have $\operatorname{d}_\mathcal{H}(w,\mathcal C)>C$.
	\end{itemize}
\end{DEF}
Denote $K=\lfloor n^\beta\rfloor$. The following lemma allows us to restrict our consideration on good sets with size $K$ in later discussions.
\begin{lemma}\label{lem:good-set}
  When $n$ is large enough, for any triple $(\pi^*,G,\G)\in \Gc$ and any subset $B\subset V$ with $|B|\ge \delta n$, there exists $A \subset B$ with $|A|=K$ such that $A$ is a good set in $\mathcal H_{\pi^*}$. In other words, for any $B\subset V$ with $|B|\ge \delta n$,
  \begin{equation}\label{eq:good-set}
  	\mathbf{1}_\Gc\le \sum_{A\subset B,|A|=K}\mathbf{1}_\Gc\mathbf{1}_{\{A\mbox{ is a good set in }\mathcal H_{\pi^*}\}}\,.
  \end{equation}
\end{lemma}
\begin{proof}
Let $A$ be a maximal good set that is contained in $B$. It suffices to show that $|A| \geq K$ (since if $|A|>K $ we can then take a subset of $A$ of cardinality $K$ which will be a good set). Let $B_0 \subset B$ be the collection of vertices $v\in B$ satisfying $\operatorname{d}_{\mathcal{H}_{\pi^*}}(v,\mathcal C)\le C$ for some $k$-cycle $\mathcal C$ in $\mathcal H_{\pi^*}$ with $k\le C$. For $v\in B$, let $R_v$ be the collection of all vertices $u\in B$ such that $\operatorname{d}_\mathcal{H}(u,v)\leq 2C+2$. On the event $\mathcal G$, we have that the maximal degree of $\mathcal H_{\pi^*}$ is bounded by $\log n$, and the number of $k$-cycles in $\mathcal H_{\pi^*}$ is bounded by $n^{\delta_1 k}$. Therefore,
\begin{equation}\label{eq-admissibility-consequence-good-set}
|B_0| \leq \sum_{k=1}^Cn^{\delta_1 k}k(\log n)^C \mbox{ and } |R_v| \leq (\log n)^{2C+2} \mbox{ for } v\in V.
\end{equation} 
Note that $B$ must be contained in
$\cup_{v\in A} R_v \cup B_0$, since otherwise one can add a vertex from $B\setminus (\cup_{v\in A}R_v\cup B_0)$ to $A$ which yields a larger good set and thus contradicts to the maximality of $A$. Combined with the assumption $|B|\ge \delta n$, the choices for $\beta, \delta_1$ and \eqref{eq-admissibility-consequence-good-set}, this implies that $|A| \geq K$ for large $n$.
\end{proof}

\subsection{Proof of Proposition~\ref{prop:reduce-to-posterior-truncation} via second moment bound}\label{subsec:proof-via-second-moment-bound}

Denote the edge likelihood ratio function as
\begin{equation*}
	\label{eq:ell}
	\ell(x,y)\stackrel{\operatorname{def}}{=}\frac{\Qc[(G_e,\G_{\Pi^*(e)})=(x,y)]}{\Ps[G_e=x]\Ps[{\G_{\Pi^*(e)}}=y]}=
	\begin{cases}
		\frac{1-2ps+ps^2}{(1-ps)^2},\ &x=y=0\,;\\
		\frac{1-s}{1-ps},\quad &x=1,y=0\text{ or }x=0,y=1\,;\\
		\frac{1}{p},&x=y=1\,.
	\end{cases}
\end{equation*}
Further, write
\begin{equation}
	\label{eq:PQR}
	P=\frac{(1-2ps+ps^2)}{p(1-s)^2}\,,\quad Q=\frac{(1-s)(1-ps)}{1-2ps+ps^2}\,,\quad R=\frac{1-2ps+ps^2}{(1-ps)^2}\,.
\end{equation}
Then it is straightforward to check that
\begin{equation}
	\label{eq:relation-dQ/dP-PQR}
	\frac{\Qc[\pi,G,\mathsf G]}{\Ps[G,\mathsf G]}=\frac{1}{n!}\frac{\Qc[G,\G\mid \pi^*=\pi]}{\Ps[G,\G]}=\frac{1}{n!}\prod_{e\in E_0}\ell(G_e,\G_{\Pi(e)})= \frac{P^{|\mathcal E_\pi|}Q^{|E|+|\mathsf E|}R^{\binom{n}{2}}}{n!}\,,
\end{equation}
where $\mathcal E_\pi$ is the set of edges in the $\pi$-intersection graph $\mathcal H_\pi$.

For any $A\subset V$ with $|A|=K$ and any $\sigma\in \B(A,\Vs)$, define
	\begin{align}
			&f(G,\G,A,\sigma)= \sum_{\pi\in \B(V,\Vs,A,\sigma)}P^{|\mathcal E_\pi|}Q^{|E|+|\Es|}R^{\binom{n}{2}}\mathbf{1}_{\Gc}\mathbf{1}_{\{A \mbox{ is good in }\mathcal H_{\pi}\}}	\label{def:f}\\
			=&\ P^{|(\mathcal E_\sigma)_A|}Q^{|E_A|+|\Es_\As|}R^{\binom{K}{2}}\sum_{\pi\in \B(V,\Vs,A,\sigma)}P^{|(\mathcal E_\pi)^{A}|}Q^{|E^A|+|\Es^\As|}R^{\binom{n}{2}-\binom{K}{2}}\mathbf{1}_{\Gc}\mathbf{1}_{\{A \mbox{ is good in }\mathcal H_{\pi}\}}\,,\nonumber
	\end{align}
where $(\mathcal E_\sigma)_A$ is the edge set of $(\mathcal H_\pi)_A$ for any $\pi\in \B(V,\Vs,A,\sigma)$ (note that this is well-defined). In addition, we set
\begin{equation}
	\label{def:g}
	g(G,\G,A)=\max_{\sigma\in \B(A,\Vs)}f(G,\G,A,\sigma)\,.
\end{equation}
\begin{proposition}
	\label{prop:further-reduction}
	For any $\tilde{\pi}\in\B(V,\Vs)$, we have 
	\begin{equation}\label{eq:bound-by-f}
		\sum_{\pi:\operatorname{overlap}(\pi,\tilde{\pi})\ge \delta n}\frac{\Qc[\pi,G,\G]\mathbf{1}_\Gc}{\Ps[G,\G]}\le \frac{1}{n!}\sum_{A\subset V,|A|=K}f(G,\G,A,\tilde{\pi}|_A)\,.
	\end{equation}
Therefore, the left hand side of \eqref{eq:reduction_2} is bounded by
\begin{equation}\label{eq:bound-by-g}
	\mathbb{E}_{(G,\G)\sim \Ps}\frac{1}{n!}\sum_{A\subset V,|A|=K}g(G,\G,A)\\,.
\end{equation}
\end{proposition}
\begin{proof}
	For any $\phi\in S_V$, denote $F(\phi)$ for the set of vertices in $V$ fixed by $\phi$. Fix some $\tilde{\pi}\in\B(V,\Vs)$ and for any $\pi\in \B(V,\Vs)$ with $\operatorname{overlap}(\pi,\tilde{\pi})\ge\delta n$, apply Lemma~\ref{lem:good-set} to the set $B=F(\pi^{-1}\circ\tilde{\pi})$ and then sum over all such $\pi$, we get that the left hand side of \eqref{eq:bound-by-f} is upper-bounded by
	\begin{equation*}
		\begin{aligned}
			&\sum_{\pi:\operatorname{overlap}(\pi,\tilde{\pi})\ge \delta n}\sum_{A\subset F(\pi^{-1}\circ\tilde{\pi}), |A|=K}\frac{\Qc[\pi,G,\G]\mathbf{1}_\Gc\mathbf{1}_{\{A \mbox{ is good in }\mathcal H_\pi\}}}{\Ps[G,\G]}\\
			\le&\ \sum_{A\subset V,|A|=K}\sum_{\pi\in \B(V,\Vs,A,\tilde{\pi}|_A)}\frac{\Qc[\pi,G,\G]\mathbf{1}_{\Gc}\mathbf{1}_{\{A\mbox{ is good in }\mathcal H_\pi\}}}{\Ps[G,\G]}
			\\
			 =&\ \frac{1}{n!}\sum_{A\subset V,|A|=K}f(G,\G,A,\tilde{\pi}|_A)\,,
		\end{aligned}
	\end{equation*}
where the last equality follows from \eqref{eq:relation-dQ/dP-PQR} and the definition of $f(G,\G,A,\sigma)$ in \eqref{def:f}. This proves \eqref{eq:bound-by-f}, and \eqref{eq:bound-by-g} follows immediately from the definition of $g(G,\G,A)$ in \eqref{def:g}.
\end{proof}
The main technical input for deriving Proposition~\ref{prop:reduce-to-posterior-truncation} is incorporated in the following proposition, whose proof is postponed to the appendix.
\begin{proposition}\label{prop:second-moment-bound}
	For any $A\subset V$ with $|A|=K$ and any $\sigma\in\B(A,\Vs)$, 
	\begin{equation*}
		\mathbb{E}_{(G,\G)\sim \Ps}f(G,\G,A,\sigma)^2\le((n-K)!)^2\exp\left(\zeta K\log K+\zeta(\alpha-1)K\log n+o(K\log n)\right)\,.
	\end{equation*}
\end{proposition}
\begin{proof}[Proof of Proposition~\ref{prop:reduce-to-posterior-truncation} assuming Proposition~\ref{prop:second-moment-bound}]
By Proposition~\ref{prop:further-reduction}, it suffice to show
\begin{equation}\label{eq:reduction_3}
	\mathbb{E}_{(G,\G)\sim \Ps}\sum_{A\subset V,|A|=K}g(G,\G,A)=o(n!)\,.
\end{equation} Applying Cauchy-Schwarz inequality to $\binom{n}{K}$ real numbers $g(G,\G,A)$ (for $A\subset V$ with $|A|=K$), we can upper-bound the left hand side of \eqref{eq:reduction_3} by
	\begin{equation*}
		\mathbb{E}_{(G,\G)\sim\Ps}\left(\binom{n}{K}\sum_{A\subset V,|A|=K}g(G,\G,A)^2\right)^{1/2}\,.
	\end{equation*}
Use the simple fact that $\max_{x\in S}x^2\le \sum_{x\in S}x^2$, the expression above is bounded by
\begin{equation*}
	\mathbb{E}_{(G,\G)\sim \Ps}\left(\binom{n}{K}\sum_{A\subset V,|A|=K}\sum_{\sigma\in \B(A,\Vs)}f(G,\G,A,\sigma)^2\right)^{1/2}\,.
\end{equation*}
By Jensen's inequality, we get that the left hide side of \eqref{eq:reduction_3} is further bounded by
\begin{equation*}
	\left(\binom{n}{K}\sum_{A\subset V,|A|=K}\sum_{\sigma\in \B(A,\Vs)}\mathbb{E}_{(G,\G)\sim \Ps}f(G,\G,A,\sigma)^2\right)^{1/2}\,.
\end{equation*}
Combined with Proposition~\ref{prop:second-moment-bound}, this  is no more than
\begin{align}
			&\left[\binom{n}{K}^3K!((n-K)!)^2\exp\left(\zeta K\log K+\zeta(\alpha-1)K\log n+o(K\log n)\right)\right]^{1/2} \nonumber\\
			=&n!\exp\left(\frac{1+\zeta(\alpha-1)-\beta(2-\zeta)}{2}K\log n+o(K\log n)\right)\,, \label{eq-K-not-too-small}
\end{align}
which is $o(n!)$ by the choice of $\beta$ as in $\eqref{eq:beta}$. This proves \eqref{eq:reduction_3} and thus completes the proof of Proposition~\ref{prop:reduce-to-posterior-truncation}.
\end{proof}

\appendix
\section{Complimentary proofs}\label{sec:appendix}
\subsection{Proof of Proposition~\ref{prop:exponential-moment-of-orbit}}\label{subsec:a-1}
\begin{proof}
	Fix $\theta>0$. Let $\{(I_k,J_k,\mathsf J_k)\}_{k=1}^\infty$ be i.i.d. triples of independent Bernoulli variables with $\E I_k = p$ and $\E J_k = \E \mathsf J_k =s$. For each integer $m\ge 1$, let $a_m,b_m,c_m$ be the value of $$\mathbb E \exp\left(\theta\sum_{i=0}^{m-1} G_{i}\mathsf G_{i+1}\right)$$ 
	where $(G_k,\mathsf G_k)=(I_kJ_k,I_k\mathsf J_k)$ for $1\le k\le m-1$ and $(G_0,\mathsf G_{m})=(0,0),(1,1),(0,1)$ respectively. 
	
	A straightforward recurrence argument yields that for any $m\ge 1$,
\begin{equation*}
	\begin{cases}
		a_{m+1}=psc_m+(1-ps)a_m\,,\\
		c_{m+1}=pse^\theta[sc_m+(1-s)a_m]+[ps(1-s)c_m+(1-2ps+ps^2)a_m]\,,
	\end{cases}
\end{equation*}
and 
\begin{equation*}
	\begin{cases}
		b_{m+1}=pse^\theta[sb_m+(1-s)c_m]+[ps(1-s)b_m+(1-2ps+ps^2)c_m]\,,\\
		c_{m+1}=psb_m+(1-ps)c_m\,.
	\end{cases}
\end{equation*}
As a result, $a_m,b_m,c_m$ can be written as linear combinations of $\mu_1^m$ and $\mu_2^m$ where $\mu_1>1>\mu_2>0$ are the two roots of the characteristic polynomial 
\begin{equation*}
	x^2-\left(1+ps^2\nu\right)x+(ps^2-p^2s^2)\nu\,,\quad\mbox{where }\nu\stackrel{\operatorname{def}}{=}e^\theta-1\,.
\end{equation*}
For any $\theta$ with $1\ll e^\theta\ll n$,  asymptotically it holds
\begin{equation}
	\label{eq:aymptotics}
	\begin{cases}
		\mu_1=1+p^2s^2\nu+p^3s^4\nu^2+O(p^4s^4\nu^2)=1+e^\theta/ n^{1+\alpha+o(1)}\,,\\
		\mu_2=ps^2\nu-p^2s^2\nu-p^3s^4\nu^2+O(p^4s^4\nu^2)=(\lambda+o(1))e^\theta /n\,.
	\end{cases}
\end{equation}
Indeed, $1\ll e^\theta\ll n$ implies $ps^2\nu\ll 1\ll \nu$, and as a result we have
\begin{align*}
	\mu_1=&\ \frac{1+ps^2\nu+\sqrt{(1+ps^2\nu)^2-4(ps^2-p^2s^2)\nu}}{2}= \frac{1+ps^2\nu+(1-ps^2\nu)\sqrt{1+\frac{4p^2s^2\nu}{(1-ps^2\nu)^2}}}{2}\\
	=&\  \frac{1+ps^2\nu+(1-ps^2\nu)\left(1+\frac{2ps^2\nu}{(1-ps^2\nu)^2}+O(p^4s^4\nu^2)\right)}{2}
	=1+p^2s^2\nu+p^3s^4\nu^2+O(p^4s^4\nu^2)\,.
\end{align*}
This yields the first equity in \eqref{eq:aymptotics} and the second follows from $\mu_2=1+ps^2\nu-\mu_1$.

Note that for any orbit $O_k\in (\mathcal O_\pi)_A$ with length $k$, 
$
\mathbb{E}_{(G,\G)\sim\Qc[\cdot\mid\pi^*]}\exp\left(\theta|\mathcal E_{O_k}|\right)
$
 is a linear combination of $a_k,b_k$ and $c_k$. More precisely, when $O$ is a $k$-cycle, 
 $$
 \mathbb{E}_{(G,\G)\sim\Qc[\cdot\mid\pi^*]}\exp\left(\theta|\mathcal E_{O_k}|\right)=(1-2ps+ps^2)a_k+ps^2b_k+2ps(1-s)c_k\,,
 $$ 
  and when $O$ is a $k$-chain,
  $$
   \mathbb{E}_{(G,\G)\sim\Qc[\cdot\mid\pi^*]}\exp\left(\theta|\mathcal E_{O_k}|\right)=(1-ps)^2a_k+p^2s^2b_k+2ps(1-ps)c_k\,.
   $$
   Hence in both cases, $ \mathbb{E}_{(G,\G)\sim\Qc[\cdot\mid\pi^*]}\exp\left(\theta|\mathcal E_{O_k}|\right)$ has the form $c_1\mu_1^k+c_2\mu_2^k$ for some constants $c_1,c_2$ determined by initial conditions. 
   
   In the cycle case, it can be shown that $c_1=c_2=1$, and this proves \eqref{eq:exp-moment-for-cycles} (see also Remark~\ref{rmk:trace}). In the chain case, we can  compute the initial values of $k=1,2$ explicitly. Then we obtain
\[
\begin{cases}
	c_1\mu_1+c_2\mu_2=1+p^2s^2\nu\stackrel{\Delta}{=}A_1\,,\\
	c_1\mu_1^2+c_2\mu_2^2=1+2p^2s^2\nu+p^3s^4\nu^2\stackrel{\Delta}{=}A_2\,.
\end{cases}
\]
Solving the linear system yields that
\begin{equation}
	\label{eq:coefficients}
		c_1= \frac{A_2-\mu_2A_1}{\mu_1(\mu_1-\mu_2)},\quad	c_2= \frac{\mu_1A_1-A_2}{\mu_2(\mu_1-\mu_2)}\\\,.
\end{equation}
Plugging \eqref{eq:aymptotics} into \eqref{eq:coefficients} then gives 
\[
\begin{aligned}
	c_1=&\  \frac{1+2p^2s^2\nu+p^3s^4\nu^2-(1+p^2s^2\nu)((ps^2-p^2s^2)\nu-p^3s^4\nu^2+O(p^4s^4\nu^2))}{(1+p^2s^2\nu+p^3s^4\nu^2+O(p^4s^4\nu^2))(1-ps^2\nu+2p^2s^2\nu+2p^3s^4\nu^2+O(p^4s^4\nu^2))}\\
	\le&\ \frac{1-ps^2\nu+3p^2s^2\nu+p^3s^4\nu^2+O(p^4s^4\nu^2)}{1-ps^2\nu+3p^2s^2\nu+3p^3s^4\nu^2+O(p^4s^4\nu^2)}\le 1\,,
\end{aligned}
\]
and 
\[
\begin{aligned}
	c_2=&\ \frac{(1+p^2s^2\nu)(1+p^2s^2\nu+p^3s^4\nu^2+O(p^4s^4\nu^2))-(1+2p^2s^2\nu+p^3s^4\nu^2+O(p^4s^4\nu^2))}{(1+o(1))ps^2\nu}\\
	=&\ \frac{O(p^4s^4\nu^2)}{(1+o(1))ps^2\nu}=O(p^3s^2e^\theta)=e^\theta n^{-1-2\alpha+o(1)}\,.
\end{aligned}
\]
This completes the proof.
\end{proof}
\begin{remark}
	\label{rmk:trace}
	Another way to see that why the coefficients must be $1$ in the cycle case is the following: as illustrated in \cite[Appendix A]{WXY21+}, one can view $\mathbb{E}_{(G,\G)\sim \Qc[\cdot\mid\pi^*]}\exp(\theta |\mathcal E_\pi(O_k)|)$ as the trace of the $k$-th power of certain integral operator $\mathcal L$ on the space of real functions on $\{0,1\}$. The spectrum of $\mathcal L$ is precisely $\{\mu_1,\mu_2\}$, so the trace of $\mathcal L^k$ equals to $\mu_1^k+\mu_2^k$. 
\end{remark}
\subsection{Proof of Lemma~\ref{lem:compute-the-minimum}}\label{subsec:a-2}
\begin{proof}
For exposition convenience, we prove a slightly stronger statment where we allow $x_0, \ldots, x_{N+1}$ to take real values instead of just integer values (but otherwise satisfy the same set of inequalities as specified in the definition of $\Sigma_T$ and $\Delta_{n_1, \ldots, n_N}$).
	A straightforward algebraic manipulation yields that
\[
x_0+\sum_{k=1}^{N+1}\alpha_k x_k=(1-\alpha_{N+1})x_0+\alpha_{N+1}\sum_{k=0}^{N+1}x_k-\sum_{m=1}^{N}(\alpha_{m+1}-\alpha_m)\sum_{k=1}^m x_k\,.
\]
As a result, the minimum of $x_0+\sum_{k=1}^{N+1} \alpha_kx_k$ for $(x_0,\dots,x_{N+1})\in \Sigma_T\cap \Delta_{n_1,\dots,n_N}$ is achieved at $x_0=0,x_k=(\varrho+\eta)kn_k,1\le k\le N$ and $x_{N+1}=[(\alpha-\eta)T-(\varrho+\eta)\sum_{k=1}^N kn_k]\vee 0$. Hence, $M(T,n_1,\dots,n_N)$ is no less than
\begin{align}
	&\ [\alpha_{N+1}(\varrho-\eta)-1]T-\sum_{k=1}^N\left[\alpha_{N+1}(\varrho+\eta)k-(k-1)(\varrho+\eta)-1\right] n_k\nonumber\\
	\ge &\ [\alpha_{N+1}(\varrho-\eta)-1]T-\sum_{k:1\le k\le N,(1-k(1-\alpha_{N+1}))(\varrho+\eta)>1}[(1-k(1-\alpha_{N+1}))(\varrho+\eta)-1]n_k\nonumber\\
	\ge&\  [\alpha_{N+1}(\varrho-\eta)-1]T-[\alpha_{N+1}(\varrho+\eta)-1]n_1\nonumber\\
	&\quad\quad\quad-\max_{2\le k \le \frac{\varrho+\eta-1}{(1-\alpha_{N+1})(\varrho+\eta)}}\frac{[(1-k(1-\alpha_{N+1}))(\varrho+\eta)-1]}{k}\times \sum_{t=2}^{N}tn_{t}\,.\label{eq:minimum-bound}
\end{align}
If $\frac{\varrho+\eta-1}{(1-\alpha_{N+1})(\varrho+\eta)}<2$, then the right hand side of \eqref{eq:minimum-bound} is lower-bounded by
\begin{equation}
	\label{eq:minimum-bound-1}
[\alpha_{N+1}(\varrho-\eta)-1]T-[\alpha_{N+1}(\varrho+\eta)-1]n_1\ge [\alpha_{N+1}(\varrho-\eta)-1- c_\lambda^{-1}(\varrho+\eta)\delta]T.
\end{equation}
If $\frac{\varrho+\eta-1}{(1-\alpha_{N+1})(\varrho+\eta)}\geq 2$, then the right hand side of \eqref{eq:minimum-bound} is lower-bounded by
\begin{align}
		&\  [\alpha_{N+1}(\varrho-\eta)-1]T-[\alpha_{N+1}(\varrho+\eta)-1]n_1-\left[\frac{\varrho+\eta-1}{2}-(1-\alpha_{N+1})(\varrho+\eta)\right]T\nonumber\\
	=&\ \left[\frac{\varrho-\eta-1}{2}-\left(2\alpha_{N+1}-\frac{1}{2}\right)\eta\right]T-[\alpha_{N+1}(\varrho+\xi)-1]n_1\nonumber\\
	\ge&\ \left[\frac{\varrho-4\eta-1}{2}- c_\lambda^{-1}(\varrho+\eta)\delta\right]T\label{eq:minimum-bound-2}\,.
\end{align}
From \eqref{eq:minimum-bound-1} and   \eqref{eq:minimum-bound-2}, we may take positive constants $\delta,\delta_0$ small enough so that
$$
0<\delta_0<[\alpha_{N+1}(\varrho-\eta)-1- c_\lambda^{-1}(\varrho+\eta)\delta]\wedge\left[\frac{\varrho-4\eta-1}{2}- c_\lambda^{-1}(\varrho+\eta)\delta\right]\,,
$$
where the right hand side above is positive for small $\delta$ due to the choice of $\eta$ in Definition~\ref{def:estimator} and $N$ in \eqref{eq:N}.
This concludes the lemma.
\end{proof}

\subsection{Proof of Proposition~\ref{prop:second-moment-bound}}\label{sec:proof-of-second-moment-bound}
In this subsection we give the proof of Proposition~\ref{prop:second-moment-bound}. Our method for controlling the truncated second moment is essentially the same as approaches in \cite[Section 3]{DD22+}, and most steps here are extracted from \cite{DD22+} with only notational changes. However, a straightforward application of arguments given in \cite{DD22+} is not enough to yield the desired bound, and for this reason we have introduced additional truncations and have considered extra combinatorial structures. Therefore, we provide a self-contained proof here for completeness, despite the fact that the proof enjoys a substantial overlap with that in \cite{DD22+}.

Recall the definition of $H_A$ and $H^A$ for a graph $H=(V,E)$ and a subset $A$ of $V$ in Section~\ref{subsec:notation-and-fact}. Throughout this section, we fix $A\subset V,|A|=K$ and $\sigma\in \B(A,\Vs)$. For any $(\pi^*,G,\G)$ with $\pi^*\in \B(V,\Vs,A,\sigma)$, define 
\begin{align*}
\Gc_A^1=\{|(\mathcal E_\sigma)_A|\le \zeta K\}\,,\quad
	\Gc_A^2=\{(\mathcal H_{\pi^*})^A\mbox{ is admissible}\}\,,\quad
	\Gc_A^3=\{A\mbox{ is a good set in }\mathcal (H_{\pi^*})^A\}\,.
\end{align*}
Then it is clear that $\Gc\subset \Gc_A^1\cap\Gc_A^2, \{A\mbox{ is a good set in }\mathcal H_{\pi^*}\} \subset \Gc_A^3$ and $\Gc_A^1$ is independent with $\Gc_A^2\cap \Gc_{A}^3$ under $\Ps$. As a result, the second moment of $f(G,\G,A,\As,\sigma)$ under $\Ps$ is bounded by
\begin{equation}\label{eq:second-moment-bound-1}
	\begin{aligned}
		\mathbb{E}_{(G,\G)\sim \Ps}f(G,&\G,A,\sigma)^2\le\  \mathbb{E}_{(G_A,\G_\As)\sim\Ps}\left(P^{|(\mathcal E_\sigma)_A|}Q^{|E_A|+|\Es_\As|}R^{\binom{K}{2}}\mathbf{1}_{\Gc_A^1}\right)^2\\
		\times &\ \mathbb{E}_{(G^A,\G^\As)\sim \Ps} \left(\sum_{\pi\in \B(V,\Vs,A,\sigma)}P^{|(\mathcal E_\pi)^{A}|}Q^{|E^A|+|\Es^\As|}R^{\binom{n}{2}-\binom{K}{2}}\mathsf{1}_{\Gc_{A}^2\cap \Gc_{A}^3}\right)^2\,.
	\end{aligned}
\end{equation}
Denote the first and second term in the expression by (I) and (II), respectively. Proposition~\ref{prop:second-moment-bound} follows immediately once we can upper-bound (I) and (II) as in the next two propositions.
\begin{proposition}\label{prop:bound-I1}
	$(\mathrm{I})\le
		\exp\left(\zeta K\log K+\zeta(\alpha-1)K\log n+o(K\log n)\right)
	$.
\end{proposition}
\begin{proposition}\label{prop:bound-I2}
	$(\mathrm{II})\le((n-K)!)^2\exp(o(K\log n))$.
\end{proposition}
We first give the proof of Proposition~\ref{prop:bound-I1}, which is a standard computation for the truncated exponential moments for binomial variables.
\begin{proof}[Proof of Proposition~\ref{prop:bound-I1}]
Recall the definition of $P,Q,R$ in \eqref{eq:PQR}. Since $Q\le 1$ and $R^{2\binom{K}{2}}=\exp(o(K))$, we just need to show that 
	\begin{equation}
		\label{eq:I_1-bound-1}
		\mathbb{E}_{(G_A,\G_\As)\sim \Ps} P^{2|(\mathcal E_\pi)_A|}\mathbf{1}_{\Gc_A^1}\le\exp(\zeta K\log K+\zeta(\alpha-1)K\log n+o(K\log n))\,.
	\end{equation}
	Since $|(\mathcal E_\pi)_A|\sim \mathbf B(\binom{K}{2},(ps)^2)$ (that is, a binomial variable with parameters $\binom{K}{2}$ and $(ps)^2$) for $(G_A,\G_\As)\sim \Ps$, the left hand side of \eqref{eq:I_1-bound-1} can be expressed as
	\begin{align}
		&
		\ \sum_{t=0}^{\lfloor\zeta K\rfloor}\binom{\binom{K}{2}}{t}P^{2t}p^{2t}s^{2t}(1-p^2s^2)^{\binom{K}{2}-t}\le \exp(o(K))\times\sum_{t=1}^{\lfloor \zeta K\rfloor}\binom{\binom{K}{2}}{t}s^{2t}\nonumber\\
		=&\ \exp(o(K))\times \binom{\binom{K}{2}}{\lfloor\zeta k\rfloor}s^{2\lfloor\zeta k\rfloor}\times\left(1+\sum_{t=1}^{\lfloor \zeta K\rfloor}
		\prod_{i=0}^{t-1}\frac{\lfloor\zeta K\rfloor-i}{(\binom{K}{2}-\lfloor\zeta K\rfloor+i)s^2}\right)\,.\label{eq:I1-bound-2}
	\end{align}
	Since $\beta>1-\alpha$ by \eqref{eq:beta}, $Ks^2\gg 1$ and the last term in \eqref{eq:I1-bound-2} is $1+o(1)$, we see that the left hand side of \eqref{eq:I_1-bound-1} equals to $\exp\left(\zeta K\log K+\zeta(\alpha-1)K\log n+o(K\log n)\right)$, as desired.
\end{proof}

The rest of this section is devoted to the proof of Proposition~\ref{prop:bound-I2}. Denote $\As$ for the image of $A$ under $\sigma$. Recall $E_0^A\subset E_0$ is the set of edges in $E_0$ not within $A$, and similarly for $\Es_0^\As$. We now define several probability measures $\Qc_{A,\As,\sigma},\Qc'_{A,\As,\sigma}, \Ps_{A,\As},\Qs_{A,\As,\sigma},\Qs'_{A,\As,\sigma}$, which will play important roles in our proof. 
\begin{DEF}
	Let $\Qc_{A,\As,\sigma}$ be the probability measure on the space  $$\Omega_{A,\As}=\{(\pi^*,G^A,\G^\As):\pi^*\in \B(V,\Vs), G^A,\Gs^\As\mbox{ are subgraphs of }(V,E_0^A),(\Vs,\Es_0^\As)\}$$ 
	defined as follow: the marginal distribution of $\pi^*$ under $\Qc_{A,\As,\sigma}$ is uniform on $\B(V,\Vs,A,\sigma)$, and conditioned on $\pi^*$, $(G^A,\G^\As)$ is obtained by deleting edges within $A$ and $\As$ from a pair of correlated \ER graphs $(G,\G)$ sampled according to $\Qc[\cdot\mid \pi^*]$. Let $\Qc'_{A,\As,\sigma}$ be the conditional law of $\Qc_{A,\As,\sigma}$ under the event $\Gc_{A}^2\cap \Gc_A^3$.
	
	Further, define $\Ps_{A,\As}$ as the law of a pair of independent \ER graphs with edge density $ps$ on $(V,E_0^A)$ and $(\Vs,\Es_0^\As)$, and $\Qs_{A,\As,\sigma},\Qs_{A,\As,\sigma}'$ as the marginal law for the last two coordinates of $\Qc_{A,\As,\sigma},\Qc_{A,\As,\sigma}'$, respectively.
\end{DEF} 

Again, it is straightforward to check that 
\[
\frac{\Qc[\pi,G^A,\G^\As]}{\Ps[G^A,\G^\As]}=\frac{1}{(n-K)!}\prod_{e\in E_0^A}\ell(G_e,\mathsf G_{\Pi(e)})=\frac{P^{|(\mathcal{E}_\pi)^A}Q^{|E^A|+|\Es^\As|}R^{\binom{n}{2}-\binom{K}{2}}}{(n-K)!}\,.
\]
Combined with the fact that for any triple $(\pi^*,G^A,\Gs^\As)\in \Omega_{A,\As}$ we have
$\Qc[\pi^*,G^A,\G^\As]\mathbf{1}_{\Gc_{A}^2\cap\Gc_{A}^3}\le \Qc_{A,\As,\sigma}'[\pi^*,G^A,\G^\As]$, it yields that the term (II) is bounded by
\[
(n-K)!\sum_{\pi|_A=\sigma}\frac{\Qc_{A,\As,\sigma}[\pi,G^A,\G^\As]\mathbf{1}_{\Gc_{A}^2\cap\Gc_{A}^3}}{\Ps_{A,\As}[G^A,\G^\As]}\le (n-K)! \frac{\Qs'_{A,\As,\sigma}[G^A,\G^\As]}{\Ps_{A,\As}[G^A,\G^\As]}\,.
\]
It suffice to bound the second moment of the conditional likelihood ratio $L'(G^A,\G^\As)\stackrel{\operatorname{def}}{=}\frac{\Qs'_{A,\As,\sigma}[G^A,\G^\As]}{\Ps_{A,\As}[G^A,\G^{\As}]}$ under $\Ps$, or equivalently, under $\Ps_{A,\As}$.  Note that the conditional law of $L'$ given $\pi^*$ is invariant of the realization of $\pi^*$. So in what follows, we may assume $\pi^*$ as certain fixed element in $\B(V,\Vs,A,\sigma)$.

Recall the definition of the set of edge orbits $\mathcal O_\pi$ in $E_0$ induced by $\phi=\pi^{-1}\circ \pi^*$. For any $\pi\in \B(V,\Vs,A,\sigma)$, let $\mathcal O_\pi^A\subset \mathcal O_\pi$ be the edge orbits that are entirely contained in $E_0^A$ and let $\mathcal J_\pi\subset \mathcal O_\pi^A$  be the set of edge orbits in $\mathcal O_\pi^A$ that are \emph{entirely} contained in $(\mathcal H_{\pi^*})^A$ (i.e., the $\pi^*$-intersecting graph of $G^A$ and $\Gs^\As$). Note that while $\mathcal O_\pi^A$ is deterministic whenever $\pi$ is fixed, $\mathcal J_\pi$ is random depending on the realization of $(G^A,\G^\As)$. Let $H(\mathcal J_\pi)$ be the subgraph of $(\mathcal H_{\pi^*})^A$ with vertices and edges from orbits in $\mathcal J_\pi$. With slight abuse of notation, we denote $|\mathcal J_\pi|$ for the total number of \emph{edges} in orbits $O\in \mathcal J_\pi$. Inspired by \cite[Lemma 3.2]{DD22+}, we have the following proposition, which connects the second moment of $L'$ with the exponential moment of $\mathcal J_\pi$. 	
\begin{proposition}
	The second moment of $L'(G^A,\G^\As)$ under $\Ps_{A,\As}$ is bounded by
	\begin{equation}
		\label{eq:bound-the-likelihoodratio-0}
		\frac{1}{\Qc_{A,\As,\sigma}[\Gc_2^A\cap \Gc_3^A]}\ 
	\mathbb{E}_{(\pi^*,G^A,\G^\As)\sim \Qc_{A,\As,\sigma}'} \frac{1}{(n-K)!}\sum_{\pi\in \B(V,\Vs,A,\sigma)}\frac{p^{-|\mathcal J_\pi|}}{\Qc[\Gc_A^2\cap \Gc_A^3\mid \pi^*,\mathcal J_\pi]}\,.
	\end{equation}
\end{proposition}
\begin{proof}
	It is clear that 
	\begin{align*}
		& \Qs_{A,\As,\sigma}'[G^A,\G^\As]=\sum_{\pi\in \B(V,\Vs,A,\sigma)}\Qc_{A,\As,\sigma}'[\pi,G^A,\G^\As]\\
		\le&\ \sum_{\pi\in\B(V,\Vs,A,\sigma)}\frac{\Qc_{A,\As,\sigma}[\pi,G^A,\G^\As]}{\Qc_{A,\As,\sigma}[\Gc_A^2\cap \Gc_A^3]}= \frac{\Qs_{A,\As,\sigma}[G^A,\G^\As]}{\Qc_{A,\As,\sigma}[\Gc_{A}^2\cap \Gc_{A}^3]}\,.
		\end{align*} 
	Thus, $L'(G^A,\G^\As)$ is bounded by $\frac{L(G^A,\G^\As)}{\Qc_{A,\As,\sigma}[\Gc_{A}^2\cap \Gc_{A}^3]}$, where $L(G^A,\G^\As)\stackrel{\operatorname{def}}{=}\frac{\Qs_{A,\As,\sigma}[G^A,\G^A]}{\Ps_{A,\As}[G^A,\G^\As]}$ is the unconditional likelihood ratio and can be written explicitly as
	\[
	\begin{aligned}
	L(G^A,\G^\As)=&\ \frac{1}{(n-K)!}\sum_{\pi\in \B(V,\Vs,A,\sigma)}\prod_{e\in E_0^A}\ell(G_e,\G_{\Pi(e)})\\
	=&\ \frac{1}{(n-K)!}\sum_{\pi\in\B(V,\Vs,A,\sigma)}\prod_{O\in \mathcal O_\pi^A}\prod_{e\in O}\ell(G_e,\G_{\Pi(e)})\,.
	\end{aligned}
	\]
	Hence, the second moment of $L'$ under $\Ps_{A,\As}$, or equivalently, the first moment of $L'$ under $\Qs_{A,\As,\sigma}'$ is bounded by
	\begin{equation}\label{eq:bound-the-likelihoodratio-1}
	\frac{1}{\Qc_{A,\As,\sigma}[\Gc_{A}^2\cap\Gc_{A}^3]}\ \mathbb{E}_{\pi^*\sim \Qc'}\frac{1}{(n-K)!}\sum_{\pi\in \B(V,\Vs,A,\sigma)}\mathbb{E}_{(G^A,\G^\As)\sim\Qc'_{A,\As,\sigma}[\cdot\mid \pi^*]}\prod_{O\in \mathcal O_\pi^A}\prod_{e\in O}\ell(G_e,\G_{\Pi(e)})\,.
	\end{equation}
	A crucial fact is that for each $\pi\in \B(V,\Vs,A,\sigma)$ and each orbit $O\in \mathcal O_\pi^A$, it holds
	\begin{equation}\label{eq:incoplete-orbit-<=1}
	\mathbb{E}_{(G^A,\G^\As)\sim \Qc_{A,\As,\sigma}[\cdot\mid \pi^*]}\left[\prod_{e\in O}\ell(G_e,\G_{\Pi(e)})\mid O\notin \mathcal J_\pi\right]\le 1\,.
	\end{equation}
In other words, for those cycles not entirely contained in $(\mathcal H_{\pi^*})^A$, their contribution to the likelihood ratio is negligible. 
The proof of \eqref{eq:incoplete-orbit-<=1} can be found in \cite[Lemma 3.1]{DD22+} via an explicit computation. Intuitively, \eqref{eq:incoplete-orbit-<=1} tells us that the main contribution of \eqref{eq:bound-the-likelihoodratio-1} comes from orbits in $\mathcal J_\pi$ and this motivates us to take conditional expectation with respect to the random set $\mathcal J_\pi$. For any two fixed $\pi^*,\pi\in\B(V,\Vs,A,\sigma)$, by averaging over conditional expectation given $\mathcal J_\pi$, 
\begin{align}\label{eq:bound-the-likelihoodratio-2}
		&\ \mathbb{E}_{(G^A,\G^\As)\sim \Qc_{A,\As,\sigma}[\cdot\mid \pi^*]}\prod_{O\in \mathcal O_\pi^A}\prod_{e\in O}\ell(G_e,\G_{\Pi(e)})\nonumber\\
		=&\ \mathbb{E}_{\mathcal J_\pi\sim \Qc'_{A,\As,\sigma}[\cdot\mid \pi^*]}\left[p^{-|\mathcal J_\pi|}\mathbb{E}_{(G^A,\G^\As)\sim \Qc_{A,\As,\sigma}'[\cdot\mid \pi^*]}\left[\prod_{O\in \mathcal O_\pi^A\setminus \mathcal J_\pi}\prod_{e\in O}\ell(G_e,\G_{\Pi(e)})\mid \mathcal J_\pi\right]\right]
\end{align}
For any realization $J$ of $\mathcal J_\pi$, conditioned on $\mathcal J_\pi=J$, we may upper-bound the expectation in the inner layer of \eqref{eq:bound-the-likelihoodratio-2} by 
\begin{equation}
	\label{eq:bound-the-likelihoodratio-1.5}
	\frac{1}{\Qc_{A,\As,\sigma}[\Gc_A^2\cap \Gc_A^3\mid{\pi^*,\mathcal J_\pi=J}]}\mathbb{E}_{(G^A,\G^\As)\sim \Qc_{A,\As,\sigma}[\cdot\mid \pi^*,\mathcal J_\pi=J]}\left[\prod_{O\in \mathcal O_\pi^A\setminus J}\prod_{e\in O}\ell(G_e,\G_{\Pi(e)})\right]\,,
\end{equation}
where the term $\frac{1}{\Qc_{A,\As,\sigma}[\Gc_A^2\cap \Gc_A^3\mid \pi^*,\mathcal J_\pi=J]}$ emerges because we move from $\Qc_{A,\As,\sigma}'[\cdot\mid\pi^*, \mathcal J_\pi=J]$ to $\Qc_{A,\As,\sigma}[\cdot\mid \pi^*,\mathcal J_\pi=J]$ in the expectation.
Furthermore, note that under the law $\Qc_{A,\As,\sigma}[\cdot\mid \pi^*,\mathcal J_\pi=J]$, for distinct orbits $O\in \mathcal O_\pi^A\setminus J$, the families of random variable $\{(G_e,\G_{\Pi(e)}):e\in O\}$ are mutually independent with law $\Qc_{A,\As,\sigma}[\cdot\mid O\notin \mathcal J_\pi]$. Thus the last expectation in \eqref{eq:bound-the-likelihoodratio-1.5} is no more than $1$ by \eqref{eq:incoplete-orbit-<=1}. Combined this with  \eqref{eq:bound-the-likelihoodratio-2} and \eqref{eq:bound-the-likelihoodratio-1.5} gives that the left hide side of \eqref{eq:bound-the-likelihoodratio-2} is bounded by
\begin{equation}
	\label{eq:bound-the-likelihoodratio-3}
	\begin{aligned}
		 \mathbb{E}_{\mathcal J_\pi\sim \Qc'_{A,\As,\sigma}[\cdot\mid \pi^*]}\frac{p^{-|\mathcal J_\pi|}}{\Qc_{A,\As,\sigma}[\Gc_A^2\cap \Gc_A^3\mid\pi^*, \mathcal J_\pi]}=\mathbb{E}_{(G^A,\G^\As)\sim \Qc'_{A,\As,\sigma}[\cdot\mid \pi^*]}\frac{p^{-|\mathcal J_\pi|}}{\Qc_{A,\As,\sigma}[\Gc_A^2\cap \Gc_A^3\mid\pi^*, \mathcal J_\pi]}\,.
	\end{aligned}
\end{equation}
Plugging \eqref{eq:bound-the-likelihoodratio-3} into \eqref{eq:bound-the-likelihoodratio-1} yields the desired result \eqref{eq:bound-the-likelihoodratio-0}.
\end{proof}
The following lemma bounds the probabilities appearing in \eqref{eq:bound-the-likelihoodratio-0} from below. 
\begin{lemma}
	\label{lem:probability-of-Gc_A^2capGc_A^3}
	There exists a constant $c_0=c_0(\lambda)>0$ such that for large enough $n$, 
	\begin{equation}
		\label{eq:probability-of-Gc_A^2capGc_A^3-uncondition}
		\Qc_{A,\As,\sigma}[\Gc_A^2\cap \Gc_A^3]\ge c_0^K\,.
    \end{equation}
    Further, for any $\pi^*,\pi\in \B(V,\Vs,A,\sigma)$ and any realization $J$ of $\mathcal J_\pi$ satisfying
    $\Qc_{A,\As,\sigma}[\Gc_A^2\cap \Gc_A^3\mid H(J)\subset (\mathcal H_{\pi^*})^A]>0$, it holds
    \begin{equation}
    	\label{eq:probability-of-Gc_A^2capGc_A^3-condition}
    	 \Qc_{A,\As,\sigma}[\Gc_A^2\cap \Gc_A^3\mid\pi^*,\mathcal J_\pi=J]\ge c_0^{K+|J|}\,.
    \end{equation}
\end{lemma} 
\begin{proof}
	Clearly, both $\Gc_A^2$ and $\Gc_A^3$ are decreasing events measurable with respect to $\mathcal H_{\pi^*}$. It can be shown in the same way as Lemma~\ref{lem:G-is-a.s.s.} that $\Qc_{A,\As,\sigma}[\Gc_A^2]=1-o(1)$, and$$\Qc_{A,\As,\sigma}[\Gc_{A}^3]\ge \Qc_{A,\As,\sigma}[\{\mbox{There is no edge between }A\mbox{ and }V\setminus A\mbox{ in }(\mathcal H_{\pi^*})^A\}]\ge c_1^K$$
	for some constant $c_1=c_1(\lambda)>0$. Applying FKG inequality to the events $\Gc_A^2$ and $\Gc_A^3$ yields that $\Qc_{A,\As,\sigma}[\Gc_A^2\cap\Gc_{A}^3]\ge (1-o(1))c_1^{K}$.
		
	Further, $\{\mathcal J_\pi=J\}$ can be written as the intersection of $\{H(J)\subset (\mathcal H_{\pi^*})^A\}$ (here $\subset$ means being a subgraph of) and a decreasing event $\mathcal G^4_{J}$ (which is measurable with respect to edges not in $E(J)$). Let $\mathcal A$ be the event that there is no edge incident to $A$ except those edges in  $E(J)$. Since $\Qc_{A,\As,\sigma}[\Gc_A^3\mid H(J)\subset(\mathcal H_{\pi^*})^A]>0$, we have 
	$
	\mathcal A\cap \{H(J)\subset (\mathcal{H}_{\pi^*})^A\} \subset \Gc_A^3 
	$. Thus,
	\begin{align*}
	\ \Qc_{A,\As,\sigma}[\Gc_A^2\cap \Gc_A^3\mid \pi^*,\mathcal J_\pi=J] &\geq\Qc_{A,\As,\sigma}[\Gc_A^2\cap \mathcal A\mid \pi^*,\mathcal J_\pi=J]\\
	& \geq \Qc_{A,\As,\sigma}[\Gc_A^2\cap \mathcal A\mid \pi^*, H(J) \subset \mathcal (H_{\pi^*})^A]\,,
	\end{align*}
	where for the second inequality we have applied FKG inequality to the conditional measure $\Qc_{A,\As,\sigma}[\cdot\mid \pi^*, H(J) \subset \mathcal (H_{\pi^*})^A]$ (which is a product measure outside of $E(J)$) and the fact that $\mathcal G^4_J$, $\Gc_A^2$ and  $\mathcal A$ remain to be decreasing events restricted to the space satisfying  $H(J) \subset \mathcal (H_{\pi^*})^A$ (i.e., in the space for the conditional measure $\Qc_{A,\As,\sigma}[\cdot\mid \pi^*, H(J) \subset \mathcal (H_{\pi^*})^A]$). Applying FKG again, we get that
	\begin{align*}
	\ \Qc_{A,\As,\sigma}[\Gc_A^2\cap \Gc_A^3\mid \pi^*,\mathcal J_\pi=J] 
	\geq  \Qc_{A,\As,\sigma}[\mathcal \Gc_A^2\mid \pi^*,H(J)\subset (\mathcal H_{\pi^*})^A]\Qc_{A,\As,\sigma}[\mathcal A\mid \pi^*,H(J)\subset (\mathcal H_{\pi^*})^A]\,.
	\end{align*}
	Also, by a similar argument as in the proof of \cite[Lemma 3.3]{DD22+} we can show that 
	$$\Qc_{A,\As,\sigma}[\Gc_A^2\mid \pi^*,H(J)\subset (\mathcal H_{\pi^*})^A]\ge (1-o(1))c_2^{|J|}$$ for some constant $c_2=c_2(\lambda)>0$ (indeed, this is where we use (iv) and (v) in admissibility).
Combined with the preceding inequality, it yields the lemma by picking $c_0=c_1c_2/2$.
\end{proof}
Note that any possible realization $J$ for $\mathcal J_\pi \sim \Qc'_{A,\As,\sigma}$ satisfies $\Qc_{A,\As,\sigma}[\Gc_A^2\cap \Gc_A^3\mid H(J)\subset (\mathcal H_{\pi^*})^A]>0$. From Lemma~\ref{lem:probability-of-Gc_A^2capGc_A^3}, we see \eqref{eq:bound-the-likelihoodratio-0} is bounded by $\exp(O(K))$ times 
\begin{equation}
	\label{eq:bound-exponential-moment}
	\mathbb{E}_{(\pi^*,G^A,\G^\As)\sim\Qc_{A,\As,\sigma}'}\frac{1}{(n-K)!}\sum_{\pi\in \B(V,\Vs,A,\sigma)}(c_0p)^{-|\mathcal J_\pi|}\,.
\end{equation}
Hence in order to prove Proposition~\ref{prop:bound-I2} it suffices to show that $\eqref{eq:bound-exponential-moment}$ is $\exp(o(K\log n))$. The idea is to bound the term in the expectation deterministically for any $\mathcal (H_{\pi^*})^A$ satisfying $\Gc_A^2\cap \Gc_A^3$. Our strategy is to enumerate all possible realization $J$ of $\mathcal J_\pi$ and control the number of $\pi$ such that $\mathcal J_\pi=J$, then sum over $J$. To this end, we need to introduce several notations. 

For two finite simple graphs $H$ and $\mathcal H$, a \emph{labeled embedding} of $H$ into $\mathcal H$ is an injective map $\iota:H\to \mathcal H$, such that $(\iota(u),\iota(v))$ is an edge of $\mathcal H$ when $(u,v)$ is an edge of $H$. Further, we define an \emph{unlabeled embedding} of $H$ into $\mathcal H$ as an isomorphic class of labeled embeddings:  two labeled embeddings $\iota_1,\iota_2:H\to\mathcal H$ is said to be equivalent if and only if there is an automorphism $\phi:H\to H$, such that $\iota_2=\iota_1\circ \phi$. Let $t(H,\mathcal H)$
 be the total number of unlabeled embedding of $H$ into $\mathcal H$. Then the total number of labeled embeddings of $H$ into $\mathcal H$ is given by $\operatorname{Aut}(H)t(H,\mathcal H)$, where $\operatorname{Aut}(H)$ denotes the number of automorphisms of $H$ to itself. For each isomorphic class of finite simple graphs $\mathcal C$, pick a representative element $H_\mathcal{C}$ of $\mathcal C$ and fix it. 
 
 Now let $\mathcal H=\mathcal (H_{\pi^*})^A=(V,(\mathcal E_{\pi^*})^A)$ be the $\pi^*$-intersection graph of $G^A$ and $\G^\As$. Let $\mathfrak{H}_\ell$  be the collection of connected subgraphs of $\mathcal H$ which contain $\ell$ vertices in $V\setminus A$ and at least one vertex in $A$, let $\mathfrak{C}_\ell$ (respectively $\mathfrak T_\ell$) be the collection of all such representatives $H_\mathcal{C}$ which are connected non-tree subgraphs (respectively trees) with $\ell$ vertices so that $t(H_\mathcal{C},\mathcal H)>0$.  Note that by definition, any two graphs in $\mathfrak{C}_\ell$ or $\mathfrak T_\ell$ are non-isomorphic, while two graphs in $\mathfrak H_\ell$ might be isomorphic. For each $H\in \mathfrak{H}_\ell$, fix a vertex $v_H\in H\cap A$ and denote by $\operatorname{Aut}(H,v_H)$ the number of  automorphisms of $H$ to itself fixing $v_H$.  
 When $\mathcal H$ is admissible, we have the following bounds on subgraph counts of $\mathfrak{H}_\ell,\mathfrak{C}_\ell$ and $\mathfrak{T}_\ell$.
 \begin{lemma}\label{lem-apply-good-set}
 	Under the condition $\Gc_A^2\cap \Gc_A^3$, we have for any $\ell\ge 2$, the following hold:
 	\begin{align}
 		&\sum_{H\in \mathfrak H_\ell}\operatorname{Aut}(H,v_H)\le 2K(2^\xi\log n)^{4\ell}\,,\label{eq:bound-H}\\
 		&\sum_{C\in \mathfrak C_\ell}\operatorname{Aut}(C)t(H,\mathcal H)\le \ell^3(2^{\xi+1}n^{\delta_1})^\ell\,,\label{eq:bound-C}\\
 		&\sum_{T\in \mathfrak T_\ell}\operatorname{Aut}(T)t(T,\mathcal H)\le n(4\log n)^{2(\ell-1)}\,.\label{eq:bound-T}
 	\end{align}
 \end{lemma}
 \begin{proof}
 \eqref{eq:bound-C} and \eqref{eq:bound-T} are proved in the same manner as in \cite[Lemma 3.5]{DD22+}. The proof of \eqref{eq:bound-H} also largely shares the same method with that in \cite[Lemma 3.5]{DD22+}. But since a modification is required, in what follows we provide a proof for \eqref{eq:bound-H}.
 	Note that $\sum_{H\in \mathfrak H_\ell}\operatorname{Aut}(H,v_H)$ is just the total number for labeled embedding of $H$ into $\mathcal H$ fixing $v_H$. Each labeled embedding of $H\in \mathfrak{H}_\ell$ into $\mathcal H$ can be constructed as follow: (1) choose a labeled spanning tree $T$ rooted in $A$ which intersects with $V\setminus A$ on $\ell$ vertices; and (2) add edges in $\mathcal H$ within the vertex set of $T$ to $T$ and get the final labeled embedding. By $\Gc_A^3$, each vertex $v\in V\setminus A$ is connected with at most one vertex in $A$, since otherwise there would be two vertices in $A$ having distance $2$ which is a contradiction to good set. Thus the size of $T$ in the first step is bounded by $2\ell$. Since each labeled tree with no more than $2\ell$ vertices can be encoded by a contour (according to depth-first search) with length no more than $4\ell$ starting from $A$, the number of choices for $T$ is bounded by $2K(\log n)^{4\ell}$ on the event $\Gc_A^2$ (recall (iii) in admissibility). When $T$ is fixed, there are no more than $2\xi\ell$ edges in $\mathcal H$ within the vertex set of $T$ by (i) in admissibility, so the number of ways for adding edges to $T$ is bounded by $2^{2\xi\ell}$. This gives the proof of \eqref{eq:bound-H}. 
 \end{proof}

For any realization $J$ of $\mathcal J_\pi$, $J$ can be decomposed into connected components, some of them intersect $A$ while others do not. We consider a \emph{configuration}  $\Lambda$ consisting of non-negative integers $r,s,t$, positive integers $l_j,x_j,m_k,y_k$ for $j\in [s],k\in [t]$ and distinct graphs $H_i\in \bigcup_{\ell\ge 2}\mathfrak{H}_\ell$ for $i\in[r]$, $C_j\in \mathfrak{C}_{l_j}$ for $j\in[s]$ and $T_k\in\mathfrak{T}_{m_k}$ for $k\in [t]$.  (There is nothing mathematically deep in \emph{configuration}, and this is mainly for notational convenience.)
For each such configuration $\Lambda$,
define $\Pi(\Lambda)$ to be the set of $\pi\in \B(V,\Vs,A,\sigma)$ such that the components of $\mathcal J_\pi$ which intersect with $A$ are exactly $H_i,i\in [r]$, and the other components consist of  $x_j$ copies of $C_j$ for $j\in[s]$ and $y_k$ copies of $T_k$ for $k\in [t]$. The next lemma provides an upper bound on $\Pi(\Lambda)$.
\begin{lemma}
	\label{lem:bound-number-of-pi}
	With aforementioned notations, we have
	\begin{equation}
		\label{eq:bound-number-of-pi}
		\begin{aligned}
		|\Pi(\Lambda)|\le&\ \left(n-K-\sum_{i\in [r]}|H_i\setminus A|-\sum_{j\in [s]}x_jl_j-\sum_{k\in [t]}y_km_k\right)!\times \prod_{i\in [r]}\operatorname{Aut}(H_i,v_{H_i})\\
		\times&\ \prod_{j\in [s]}\left(\operatorname{Aut}(C_j)t(C_j,\mathcal H)\right)^{x_j}\times \prod_{k\in[t]}\left(\operatorname{Aut}(T_k)t(T_k,\mathcal H)\right)^{y_k}\,.
		\end{aligned}
	\end{equation}
\end{lemma}
\begin{proof}
	Let $H_\Lambda$ be a graph which is the disjoint union of $H_i$ for $i\in [r]$, $x_j$ copies of $C_j$ for $j\in [s]$ and $y_k$ copies of $T_k$ for $k\in [t]$. 
	First we choose an unlabeled embedding $\iota:H_\Lambda\to \mathcal H$ with $\iota|_{H_i}=\operatorname{id},\forall i\in [r]$ and $\iota(H_\Lambda\setminus\bigcup_{i\in [r]} H_i)\cap A=\emptyset$. The number of such choices is bounded by 
	\[
	\prod_{j\in [s]}\binom{t(C_j,\mathcal H)}{x_j}\prod_{k\in[t]}\binom{t(T_k,\mathcal H)}{y_k}\le\prod_{j\in [s]}\frac{t(C_j,\mathcal H)^{x_j}}{x_j!}\prod_{k\in [t]}\frac{t(T_k,\mathcal H)^{y_k}}{y_k!}\,.
	\]
	
	For fixed $\iota$, any $\pi\in\B(V,\Vs,A,\sigma)$ such that $\mathcal J_\pi=\iota(H_\Lambda)$ can be decomposed into two permutations $\pi_1$ and $\pi_2$ on $A\cup \iota(H_\Lambda)$ and $V\setminus(A\cup\iota(H_\Lambda))$, respectively.  The number of choices for $\pi_2$ is at most $$\left(n-K-\sum_{i\in [r]}|H_i\setminus A|-\sum_{j\in [s]}x_jl_j-\sum_{k\in [t]}y_km_k\right)!\,.$$
	For $\pi_1$, we have the following crucial observation: for any $u\in\iota(H_\Lambda)$, $\pi_1$ inhibits to an isomorphism between the components of $\iota(H_\Lambda)$ containing $u$ and $\pi_1(u)$. That is to say, for any $u,v\in\iota( H_\Lambda)$, whenever $v$ is adjacent to $u$, $\pi_1(v)$ is adjacent to $\pi_1(u)$ and $\pi_1^{-1}(v)$ is adjacent to $\pi_1^{-1}(u)$. This is true because by definition, the edge orbit containing $(u,v)$ is entirely contained in $\mathcal J_\pi$. With such observation and the condition that $\pi_1|_A=\sigma$, we see that $\pi_1$ inhibits to an automorphism of $H_i$ to itself fixing $v_{H_i}$ for each $i\in [r]$. In addition, for any $C_j$ (and similarly for $T_k$), we will ``permute'' its $x_j$ copies of $C_j$ such that $\pi_1$ maps one copy to its image under the permutation, and within each copy we also have the freedom of choosing an arbitrary automorphism of $C_j$. Since the choices for the permutations and automorphisms determine $\pi_1$, we conclude from this that the number of choices for $\pi_1$ is no more than
	\begin{equation*}
	\prod_{i\in [r]}\operatorname{Aut}(H_i,v_{H_i})\prod_{j\in [s]}x_j!\operatorname{Aut}(C_j)^{x_j}\prod_{k\in [t]}y_k!\operatorname{Aut}(T_k)^{y_k}\,.
	\end{equation*}
These bounds altogether yield \eqref{eq:bound-number-of-pi}. 
\end{proof}

We need yet another lemma which bounds the number of edges in $H\in \mathfrak{H}_\ell$. 
\begin{lemma}
	\label{lem:apply-good-set-2}
	There exists some constant $\delta_2>0$, such that for any constant $c>0$ and $n$ large enough, under the condition $\Gc_A^2\cap\Gc_A^3$ it holds 
	\begin{equation}
		\label{eq:the-use-of-Gc_A^3}
		\sup_{H\in \mathfrak{H}_\ell}\frac{(2^\xi \log n)^{4\ell}}{n^\ell (cp)^{|E(H)|}}\le n^{-\delta_2\ell},\forall \ell\ge 1\,.
	\end{equation}
\end{lemma}

\begin{proof}
We remark that the proof of this lemma illustrates the motivation for defining good set.

For any $\ell\ge 1$ and $H\in \mathfrak H_\ell$, denote $I=H\setminus A$ and let $J\subset I$ be the set of vertices in $I$ with an edge connecting  to $A$ in $H$. Then $|I|=\ell$ by definition of $\mathfrak H_\ell$, and any vertex in $J$ is adjacent to exactly one vertex in $A$ as before. Further, we have the observation that for any two vertices $u,v\in J$, the $C$-neighborhoods of $u,v$ are disjoint (since otherwise there would be two vertices in $A$ with graph distance no more than $2C+2$, which contradicts with $\Gc_A^3$). Our proof of \eqref{eq:the-use-of-Gc_A^3} is divided into three cases.
	\begin{itemize}
		\item \emph{Case\ 1: $1\le \ell\le C$.} In this case, $H$ must be a tree, since otherwise there would be a cycle with length no more than $C$ and within distance $C$ from $A$, which contradicts with $\Gc_A^3$. Further, $J$ must be a singleton by the previous observation. Hence $E(H)\le \ell$ for any $H\in\mathfrak H_\ell$ and so
		\begin{equation}
			\label{eq:case1}
			\sup_{H\in \mathfrak{H}_\ell}\frac{(2^\xi \log n)^{4\ell}}{n^\ell (cp)^{|E(H)|}}\le n^{(\alpha-1+o(1))\ell}\,.
		\end{equation}
		\item \emph{Case\ 2: $C<\ell\le n/\log n$.} In this case, the $C$-neighborhood of each $v\in H$ contains at least $C$ vertices, so $|J|\le \ell/C$. In addition, the edges within $I$ is bounded by $\zeta \ell$ from (ii) of admissibility, so $|E(H)|\le (\zeta+C^{-1})\ell$ for any $H\in \mathfrak H_\ell$. Thus
		\begin{equation}
			\label{eq:case2}
			\sup_{H\in \mathfrak{H}_\ell}\frac{(2^\xi \log n)^{4\ell}}{n^\ell (cp)^{|E(H)|}}\le n^{(\alpha(\zeta+C^{-1})-1+o(1))\ell}\,.
		\end{equation}
		\item \emph{Case\ 3: $\ell\ge n/\log n$.} In this case, the number of edges within $H$ is bounded by $\xi|H|\le\xi(\ell+K)$ from (i) in admissibility, so
		\begin{equation}
			\label{eq:case3}
			\sup_{H\in \mathfrak{H}_\ell}\frac{(2^\xi \log n)^{4\ell}}{n^\ell (cp)^{|E(H)|}}\le n^{\alpha\xi(K+\ell)-\ell}\,.
		\end{equation}
	\end{itemize}
	By the choice of $\xi,\zeta$ and $C$ in \eqref{eq:xi}, \eqref{eq:zeta}, \eqref{eq:C}, we may choose a positive constant $\delta_2<(1-\alpha(\zeta+C^{-1}))\wedge (1-\alpha\xi)$. Then \eqref{eq:the-use-of-Gc_A^3} follows from \eqref{eq:case1}, \eqref{eq:case2} and \eqref{eq:case3} (note that in the last case $\ell \gg K$). This completes the proof.
\end{proof}

We are now ready to state and prove the final proposition in this section. 
\begin{proposition}
	\label{prop:bound-exp-moment-deterministically}
	Whenever $(\pi^*,G^A,\G^\As)$ is sampled from $\Qc_{A,\As,\sigma}'$, it holds that
	\begin{equation}\label{bound-of-exp-moment-deternministically}
		\frac{1}{(n-K)!}\sum_{\pi\in \B(V,\Vs,A,\sigma)}(c_0p)^{-|\mathcal J_\pi|}=\exp(o(K\log n))\,.
	\end{equation}
\end{proposition}
Provided with Proposition~\ref{prop:bound-exp-moment-deterministically}, we have that \eqref{eq:bound-exponential-moment} is $\exp(o(K\log n))$ and thus Proposition~\ref{prop:bound-I1} follows.
\begin{proof}
For $(\pi^*,\G^A,\Gs^\As) \sim\Qc_{A,\As,\sigma}'$, the $\pi^*$-intersection graph $\mathcal H$ satisfies $\Gc_A^2\cap \Gc_A^3$. We have
	\begin{equation}
		\label{eq:bound-exp-moment-1}
		\frac{1}{(n-K)!}\sum_{\pi\in\B(V,\Vs,A,\sigma)}(c_0p)^{-|\mathcal J_\pi|}\le\frac{1}{(n-K)!} \sum_{\Lambda}(c_0p)^{-|E(H_\Lambda)|}\times |\Pi(\Lambda)| \,,
	\end{equation}
where the sum is taken over all possible configurations $\Lambda$. It is clear that each tree $T\in \bigcup \mathfrak{T}_\ell$ has $|T|-1$ edges and each $C\in \bigcup\mathfrak C_{\ell}$ has no more than $\xi |C|$ edges by $\Gc_A^2$ and (i) in admissibility. From this and Lemma~\ref{lem:bound-number-of-pi}, we can upper bound \eqref{eq:bound-exp-moment-1} by
\begin{align*}
	&\ \frac{1}{(n-K)!}\sum_{r,s,t\ge 0}\sum_{H_1,\cdots,H_r\in \bigcup \mathfrak H_\ell}\sum_{C_1,\cdots,C_s\in \bigcup \mathfrak{C}_\ell}\sum_{T_1,\cdots,T_t\in \bigcup \mathfrak T_\ell}\sum_{x_1,\cdots,x_s>0}\sum_{y_1,\cdots,y_t>0}\\
	&\quad\ \ (c_0p)^{-\sum_{i\in [r]}|E(H_i)|-\xi\sum_{j\in [s]}x_j|C_j|-\sum_{k\in [t]}y_k(|T_k|-1)}\\
	&\ \times \left(n-K-\sum_{i\in [r]}|H_i\setminus A|+\sum_{j\in [s]}x_jl_j+\sum_{k\in [t]}y_km_k\right)!\\
	&\ \times \prod_{i\in [r]}\operatorname{Aut}(H_i,v_{H_i})
	\times
	 \prod_{j\in [s]}\left(\operatorname{Aut}(C_j)t(C_j,\mathcal H)\right)^{x_j}\times \prod_{k\in[t]}\left(\operatorname{Aut}(T_k)t(T_k,\mathcal H)\right)^{y_k}\,.
\end{align*}
Note that by Stirling's formula, with $K=\lfloor n^\beta\rfloor$ we have $(n-K-t)!/(n-K)!\le (2e/n)^t$ for any $t\ge 0$, thus for $c_1=c_0/2e$, the expression above is bounded by
\begin{align}
	\nonumber&\qquad\qquad\sum_{r,s,t\ge 0}\sum_{H_1,\cdots,H_r\in \bigcup \mathfrak H_\ell}\sum_{C_1,\cdots,C_s\in \bigcup \mathfrak{C}_\ell}\sum_{T_1,\cdots,T_t\in \bigcup \mathfrak T_\ell}\sum_{x_1,\cdots,x_s>0}\sum_{y_1,\cdots,y_t>0}\\
	\nonumber\prod_{i\in [r]}&n^{-|H_i\setminus A|}(c_1p)^{|E(H_i)|}\operatorname{Aut}(H_i,v_{H_i})\prod_{j\in [s]}(c_1np)^{-\xi x_j|C_j|}\operatorname{Aut}(C_j)^{x_j}\prod_{k\in [t]}(c_1np)^{-y_k(|T_k|-1)}\operatorname{Aut}(T_k)^{y_k}\\
	 =&\ \prod_{H\in \bigcup\mathfrak{H}_\ell}\left(1+\frac{\operatorname{Aut}(H,v_H)}{n^{|H\setminus A|}(c_1p)^{|E(H)|}}\right)\prod_{C\in\bigcup \mathfrak C_\ell}\sum_{x\ge0}\left(\frac{\operatorname{Aut}(C)}{(c_1np)^{\xi|C|}}\right)^x\prod_{T\in \bigcup \mathfrak T_\ell}\sum_{y\ge 0}\left(\frac{p\operatorname{Aut}(T)}{(c_1np)^{|T|}}\right)^y\,.\label{eq:bound-exp-moment-2}
\end{align}
As in the proof of \cite[Proposition 3.6]{DD22+}, it can be shown by \eqref{eq:bound-C} and \eqref{eq:bound-T} that the last two terms in \eqref{eq:bound-exp-moment-2} are $1+o(1)$, and thus it remains to show the first term is $\exp(o(K\log n))$. Since $\log(1+x)\le x$ for any $x\ge 0$, the logarithm of the first term in \eqref{eq:bound-exp-moment-2} is bounded by
\begin{equation}
	\label{eq:bound-exp-moment-3}
	\begin{aligned}
&\ \sum_{\ell\ge 1}\sum_{H\in \mathfrak{H}_\ell}\frac{\operatorname{Aut}(H,v_H)}{n^\ell(c_1p)^{|E(H)|}}\le  \sum_{\ell\ge 1}\sup_{H\in \mathfrak H _\ell}\frac{1}{n^\ell (c_1p)^{|E(H)|}}\sum_{H\in \mathfrak H_\ell}\operatorname{Aut}(H,v_H)\\
\stackrel{\eqref{eq:bound-H}}{\le}&\ K\sum_{\ell \ge 1}\sup_{H\in \mathfrak H_\ell}\frac{(2^\xi \log n)^{4\ell}}{n^\ell(c_1p)^{|E(H)|}}\stackrel{\text{\eqref{eq:the-use-of-Gc_A^3}}}{\le}K\sum_{\ell\ge 1} n^{-\delta_2\ell}=o(K\log n)\,,
\end{aligned}
\end{equation}
as desired.
\end{proof}

\end{document}